\documentclass[smallextended,envcountsect,envcountsame,runningheads]{svjour3}
\usepackage{amsmath}
\usepackage{amsfonts}
\usepackage{amssymb}
\usepackage{graphicx}
\usepackage{booktabs}
\usepackage{algorithm}
\usepackage{algorithmic}
\usepackage{placeins}
\usepackage[hidelinks]{hyperref}

\smartqed

\title{Random-subspace Ellipsoidal Trust-region Methods for Derivative-free Optimization\thanks{\textbf{Funding:} Chen's research is partially funded by the Killam Doctoral Scholarship and the UBC Okanagan Distinguished Doctoral Scholar Award. Xie's research is partially supported by Laboratory Directed Research and Development (LDRD) funding from Lawrence Berkeley National Laboratory and by the U.S. Department of Energy, Office of Science, Office of Advanced Scientific Computing Research applied mathematics program (SEAZOTIE) under Contract Number DE-AC02-05CH11231.}}
\titlerunning{Random-subspace Ellipsoidal Optimization}
\author{Yiwen Chen \and Pengcheng Xie}
\authorrunning{Yiwen Chen and Pengcheng Xie}
\institute{Yiwen Chen \at
    Department of Mathematics, University of British Columbia, Kelowna, British Columbia, V1V 1V7, Canada. \\
    \email{yiwchen@student.ubc.ca}
    \and
    Pengcheng Xie \at
    Applied Mathematics and Computational Research Division, Lawrence Berkeley National Laboratory. \\
    \email{pxie@lbl.gov}
}
\journalname{}
\date{}

\begin{document}

\maketitle

\begin{abstract} 
    Random-subspace model-based methods offer a promising approach to solving derivative-free optimization (DFO) problems in high dimensions. However, the usual ratio of actual to predicted reduction assesses model agreement at the trial point without directly measuring subspace quality. An accurate prediction in a poorly aligned subspace may therefore trigger trust-region expansion in all directions. We introduce Random-subspace Ellipsoidal TRust-region Optimization (RETRO), a framework for unconstrained DFO using ellipsoidal trust regions. Its trust-region updates combine global size changes with additional expansion or contraction within the sampled subspace, reflecting where the trial step was computed. We define an ellipsoidal subspace alignment measure that accounts for both the subspace and trust-region shape, and analyze Haar-uniform sampling and metric-guided sampling based on the current shape matrix. We establish almost-sure first-order liminf convergence and an $\mathcal{O}(\epsilon^{-2})$ bound on the expected number of iterations required to reach an iterate with gradient norm at most $\epsilon$. Numerical experiments on projected low-rank problems in dimension $1000$ show that using ellipsoidal shape information in subspace selection can reduce function evaluation costs relative to ball-based QARSTA, with the most pronounced improvements occurring when the solver subspace dimension matches the projection rank.
\end{abstract}

\keywords{Derivative-free optimization \and Large-scale optimization \and Randomized subspace methods \and Ellipsoidal trust regions}
\subclass{90C56 \and 65K05 \and 90C06}

\section{Introduction}
    We consider optimization problems of the form
    \begin{equation*}
        \min_{\boldsymbol{x}\in\mathbb{R}^n} f(\boldsymbol{x}),
    \end{equation*}
    where $f:\mathbb{R}^n\to\mathbb{R}$ is smooth and evaluated through a black box. The black box may be a computer simulation, a laboratory experiment, or a computational pipeline for which reliable derivatives are unavailable. Derivative-free optimization (DFO) methods address this setting using function evaluations alone~\cite{audet2017derivative,conn2009introduction,larson2019derivative}, with applications ranging from energy systems to materials science and engineering design~\cite{alarie2021two}. One major class of DFO methods, called model-based methods, constructs models to approximate objective and constraint functions~\cite{audet2020model,larson2019derivative,liuzzi2019trust}. However, constructing these models in the full space $\mathbb{R}^n$ becomes costly as the dimension $n$ increases. Indeed, determined linear and quadratic interpolation models in $\mathbb{R}^n$ require $n+1$ and $(n+1)(n+2)/2$ sample points, respectively~\cite{conn2009introduction}.

    Random subspaces offer a promising way to reduce the cost of model construction. At each iteration, a model is constructed and minimized in a low-dimensional subspace, which changes as the method progresses. This approach has been developed within the classical trust-region framework~\cite{conn2000trust}. Cartis and Roberts~\cite{cartis2023scalable} introduced a general random-subspace model-based trust-region framework with convergence and complexity guarantees and developed a practical method for nonlinear least-squares problems. QARSTA~\cite{chen2024qfully} constructs quadratic interpolation models for general smooth objectives, while subsequent work studies second-order convergence~\cite{cartis2026randomized}. Related developments address stochastic objectives~\cite{dzahini2024stochastic} and convex constraints through CLARSTA~\cite{chen2026clarsta}. The convergence of these methods relies on both the accuracy of the model within the selected subspace and the ability of the subspace to retain useful directions of descent.

    However, the usual trust-region update does not directly account for the quality of the selected subspace. In classical trust-region methods, step acceptance and radius updates are guided by the ratio of actual to predicted reduction~\cite{conn2000trust}. For a subspace model $\widehat m_k$ and a step $\boldsymbol{s}_k=\boldsymbol{Q}_k\widehat{\boldsymbol{s}}_k$, this ratio is
    \begin{equation*}
        \rho_k= \frac{f(\boldsymbol{x}_k)-f(\boldsymbol{x}_k+\boldsymbol{s}_k)}
         {\widehat m_k(\boldsymbol{0})-\widehat m_k(\widehat{\boldsymbol{s}}_k)},
    \end{equation*}
    where the columns of $\boldsymbol{Q}_k$ form an orthonormal basis of the sampled subspace. A value of $\rho_k$ close to one indicates that the model predicts the reduction at the trial point accurately. However, a subspace nearly orthogonal to the gradient may offer little descent even when that descent is predicted perfectly. Thus, $\rho_k=1$ can occur for a poorly aligned subspace. Moreover, model accuracy within that subspace provides no guarantee of accuracy outside it. Nevertheless, a sufficiently large $\rho_k$ can lead to an increase in the single radius that controls the trust region in every direction. In this sense, a successful iteration may give misleading confidence in the trust-region size in directions that the trial step has not tested.  This observation motivates updating the trust region differently in different directions, according to where the step was computed.

    To address this issue, we introduce the framework Random-subspace Ellipsoidal TRust-region Optimization (RETRO). RETRO replaces ball trust regions with ellipsoidal trust regions of the form
    \begin{equation*}
        \mathcal{T}_k=\{\boldsymbol{x}_k+\boldsymbol{s}:\boldsymbol{s}^{\top}
      \boldsymbol{S}_k\boldsymbol{s}\le\Delta_k^2\},
    \qquad \boldsymbol{S}_k\succ0.
    \end{equation*}
    The scalar $\Delta_k$ controls the overall size, and $\boldsymbol{S}_k$ controls the shape. Trial steps are computed within the intersection of this ellipsoid and the sampled affine subspace. The update combines a global change in size with an additional expansion or contraction within the sampled subspace. Thus, RETRO retains the usual acceptance ratio while allowing the trust region to change differently within the current subspace and in its orthogonal complement.

    Related concerns about interpreting successful iterations arise in probabilistic and stochastic trust-region methods. Bandeira, Scheinberg, and Vicente~\cite{bandeira2014convergence} allow models to satisfy accuracy requirements only probabilistically. Chen, Menickelly, and Scheinberg~\cite{chen2018stochastic} additionally use stochastic function estimates. Both supplement the ratio test with a model-gradient test before accepting a step and expanding the region. Cao, Berahas, and Scheinberg~\cite{cao2024first} use a relaxed ratio to accommodate noisy evaluations and a model-gradient test to control radius changes after acceptance. In these works, the additional tests account for uncertainty in the models or function evaluations. Our focus is on the information available when a step is restricted to a subspace: even with exact function values and an accurate subspace model, a successful step provides no direct information about directions outside that subspace. RETRO uses this distinction to update the trust region differently within the sampled subspace and outside it.

    Ellipsoidal trust regions also occur in scaled trust-region methods.  More recently, Hu et al. proposed MATRO \cite{hu2026matro}, which uses the Hessian of a full-space quadratic interpolation model to determine the trust-region shape. RETRO also allows different step lengths in different directions, but its shape update depends on the iteration outcome and the sampled subspace. Thus, while MATRO uses estimated curvature to adapt the trust region, RETRO uses the information obtained from successive subspace steps without requiring a full-space Hessian approximation.

    To analyze RETRO, we introduce an ellipsoidal subspace alignment measure that accounts for both the selected subspace and the shape of the trust region. We study Haar-uniform sampling and metric-guided sampling, which uses the current shape matrix when selecting the next subspace. We prove almost-sure first-order liminf convergence of RETRO and establish an $\mathcal{O}(\epsilon^{-2})$ bound on the expected number of iterations required to find an iterate with $\|\nabla f(\boldsymbol{x}_k)\|\le\epsilon$. Numerical experiments on projected low-rank problems show that RETRO can use the directional information stored in the ellipsoid to reduce function evaluation costs relative to ball-based QARSTA. The greatest improvements occur when the solver subspace dimension matches the projection rank, while tests with larger projection ranks clarify the limits of this benefit.

    The remainder of this paper is organized as follows. We end this section by introducing the notation used in the paper. Section~\ref{sec:framework} introduces ellipsoidal trust regions and presents RETRO, including model construction, step computation, and the rules for updating the size and shape of the trust region. Section~\ref{sec:convergcomplexity} establishes convergence and complexity results. It also analyzes Haar-uniform and metric-guided sampling and verifies the required subspace alignment properties. Section~\ref{sec:numexp} compares practical RETRO variants with ball-based QARSTA on projected low-rank problems. Section \ref*{sec:conclusion} concludes this paper.

\subsection{Notation}
    We write vectors in lowercase boldface and matrices in uppercase boldface (e.g., $\boldsymbol{x}\in\mathbb{R}^n$ and $\boldsymbol{A}\in\mathbb{R}^{n\times z}$).  For a symmetric matrix $\boldsymbol{A}\in\mathbb R^{n\times n}$, we use $\boldsymbol{A}\succ0$ and $\boldsymbol{A}\succeq 0$ to mean that $\boldsymbol{A}$ is positive definite and positive semi-definite, respectively.  For symmetric matrices $\boldsymbol{A}, \boldsymbol{B}\in\mathbb R^{n\times n}$, the notation $\boldsymbol{A}\succeq \boldsymbol{B}$ means that $\boldsymbol{A}-\boldsymbol{B}$ is positive semi-definite.  We use $\boldsymbol{0}_n$ and $\boldsymbol{1}_n$ to denote the all-zero and all-one vectors in $\mathbb R^n$, respectively.  We use $\boldsymbol{e}^i$ to denote the $i$-th coordinate vector with appropriate dimension.  The identity matrix of order $n$ is denoted by $\boldsymbol{I}_n$.  Functions $\mathrm{span}(\cdot)$ and $\mathrm{col}(\cdot)$ denote linear span and column space, respectively, while $\lambda_{\min}(\cdot)$ and $\lambda_{\max}(\cdot)$ denote the smallest and largest eigenvalues of a symmetric matrix.  The symbol $\|\cdot\|$ denotes the Euclidean norm of a vector and the spectral norm of a matrix, and $\|\cdot\|_F$ denotes the Frobenius norm.  We use $\mathcal{B}(\boldsymbol{x};\Delta):=\{\boldsymbol{y}:\|\boldsymbol{y}-\boldsymbol{x}\|\le \Delta\}$ to denote the closed Euclidean ball centered at $\boldsymbol{x}$ with radius $\Delta$. For a nonempty bounded set $D\subseteq\mathbb R^n$, $\mathrm{radi}(D)$ denotes the radius of the smallest Euclidean ball containing $D$.  We use $\sim$ to specify the distribution of a random variable and $\overset{d}{=}$ to indicate that two random quantities have the same distribution.

    We use $\mathcal C^{1+}$ for continuously differentiable functions with Lipschitz continuous gradient, and $\mathcal C^{2+}$ for twice continuously differentiable functions with Lipschitz continuous Hessian.  The subscript $k$ denotes the iteration index; when discussing a generic iteration, we omit this subscript.

\section{Ellipsoidal random-subspace trust-region framework}\label{sec:framework}
    RETRO constructs a model in a sampled subspace and computes a trial step within an ellipsoidal trust region.  The subspace determines the directions available to the step, while the ellipsoid determines how far the step can move in those directions.  Both therefore affect the decrease that can be obtained from the model.  We first describe this geometry and introduce an alignment measure that accounts for the subspace and the ellipsoid together.  We then present the RETRO framework and explain the model construction, step computation, and ellipsoid updates used in each iteration.

\subsection{Ellipsoidal geometry}
    A ball trust region allows steps of the same maximum length in every direction.  An ellipsoidal trust region allows these lengths to differ, so that the region can be enlarged or reduced more in some directions than in others.  We describe its size and shape separately by replacing the usual ball with
    \begin{equation*}
        \mathcal T(\boldsymbol{x};\boldsymbol{S};\Delta)
        :=\{\boldsymbol{y}\in\mathbb R^n: \left(\boldsymbol{y}-\boldsymbol{x}\right)^\top\boldsymbol{S}\left(\boldsymbol{y}-\boldsymbol{x}\right)\le \Delta^2\},
    \end{equation*}
    where $\Delta>0$ is a scalar size parameter and $\boldsymbol{S}\succ0$ is the shape matrix.  For fixed $\boldsymbol{S}$, changing $\Delta$ scales all directional radii by the same factor.  In particular, if $\boldsymbol{S}=\boldsymbol{I}_n$, then $\mathcal{T}(\boldsymbol{x};\boldsymbol{S};\Delta)$ is the usual ball trust region $\mathcal{B}(\boldsymbol{x};\Delta)$.  For a general shape matrix, $\Delta$ need not equal a Euclidean radius.  Indeed, for every unit
    vector $\boldsymbol{u}\in\mathbb R^n$,
    \begin{equation*}
        \mathrm{radi}(\left(\mathcal T(\boldsymbol{x};\boldsymbol{S};\Delta)-\boldsymbol{x}\right)\cap\mathrm{span}(\boldsymbol{u}))
        =\frac{\Delta}{\sqrt{\boldsymbol{u}^\top\boldsymbol{S}\boldsymbol{u}}}.
    \end{equation*}
    Thus a smaller value of $\boldsymbol{u}^\top\boldsymbol{S}\boldsymbol{u}$ permits a longer step along $\boldsymbol{u}$.  To compare the ellipsoid with Euclidean balls, we use its inner and outer radii.  Specifically, $\mathcal{B}(\boldsymbol{x};r^{\mathrm{in}})
        \subseteq\mathcal T(\boldsymbol{x};\boldsymbol{S};\Delta)
        \subseteq \mathcal{B}(\boldsymbol{x};r^{\mathrm{out}})$ where
    \begin{equation*}
        r^{\mathrm{in}}
        :=\frac{\Delta}{\sqrt{\lambda_{\max}(\boldsymbol{S})}}~~\text{and}~~r^{\mathrm{out}}:=\frac{\Delta}{\sqrt{\lambda_{\min}(\boldsymbol{S})}}.
    \end{equation*}

    The algorithm computes trial points in a $p$-dimensional affine subspace $\mathcal Y:=\boldsymbol{x}+\mathrm{col}(\boldsymbol{Q})$, where $1\le p\le n$ and $\boldsymbol{Q}\in\mathbb R^{n\times p}$ has orthonormal columns. Feasible trial points lie in the ellipsoidal section $\mathcal{T}(\boldsymbol{x};\boldsymbol{S};\Delta)\cap\mathcal{Y}$, which in reduced coordinates is
    \begin{equation*}
        \mathcal T(\boldsymbol{0}_p;\boldsymbol{Q}^\top\boldsymbol{S}\boldsymbol{Q};\Delta)=\{\widehat{\boldsymbol{y}}\in\mathbb R^p:\widehat{\boldsymbol{y}}^\top\boldsymbol{Q}^\top\boldsymbol{S}\boldsymbol{Q}\widehat{\boldsymbol{y}}\le\Delta^2\}.
    \end{equation*}
    Similarly, we have $\mathcal{B}(\boldsymbol{0}_p;\widehat{r}^{\mathrm{in}})\subseteq\mathcal T(\boldsymbol{0}_p;\boldsymbol{Q}^\top\boldsymbol{S}\boldsymbol{Q};\Delta)\subseteq  \mathcal{B}(\boldsymbol{0}_p;\widehat{r}^{\mathrm{out}})$ where
    \begin{equation*}
        \widehat{r}^{\mathrm{in}}
        :=\frac{\Delta}{\sqrt{\lambda_{\max}(\boldsymbol{Q}^\top\boldsymbol{S}\boldsymbol{Q})}}~~~\text{and}~~~\widehat{r}^{\mathrm{out}}:=\frac{\Delta}{\sqrt{\lambda_{\min}(\boldsymbol{Q}^\top\boldsymbol{S}\boldsymbol{Q})}}.
    \end{equation*}
    Since $\lambda_{\min}(\boldsymbol{S})\le\lambda_{\min}(\boldsymbol{Q}^\top\boldsymbol{S}\boldsymbol{Q})\le\lambda_{\max}(\boldsymbol{Q}^\top\boldsymbol{S}\boldsymbol{Q})\le\lambda_{\max}(\boldsymbol{S})$, we have $r^{\mathrm{in}}\le\widehat{r}^{\mathrm{in}}\le\widehat{r}^{\mathrm{out}}\le r^{\mathrm{out}}$.

    The shape matrix may change at every iteration.  We bound its eigenvalues uniformly so that the directional radii remain comparable to $\Delta_k$, even as the ellipsoid changes shape.  These bounds also apply to every reduced section and will allow us to express the convergence estimates in terms of the scalar size parameter.  In particular, we suppose that there are constants $\underline{\lambda}_S, \overline{\lambda}_S>0$ independent of $k$ such that
    \begin{equation}\label{ineq:shape-bounds-Sk}
        0\prec\underline\lambda_S\boldsymbol{I}_n
        \preceq\boldsymbol{S}_k
        \preceq\overline\lambda_S\boldsymbol{I}_n~~~\text{for all }k.
    \end{equation}
    Notice that the pair $(\boldsymbol{S}_k,\Delta_k)$ is not a unique description of the geometry: replacing it by $(a\boldsymbol{S}_k,\sqrt a\Delta_k)$, for $a>0$, leaves the ellipsoid unchanged. RETRO starts from a specified pair $(\boldsymbol{S}_0,\Delta_0)$ and updates its two components according to the rules in Subsection~\ref{subsec:ellipsoidupdate}, while maintaining the eigenvalue bounds above. Thus, although different pairs can describe the same ellipsoid, the algorithm uses the particular pair produced by these updates.

\subsection{Ellipsoidal subspace alignment}
    A useful subspace should contain directions along which sufficient decrease in the objective can be achieved.  For a ball trust region, this can be measured by the norm of the gradient projected onto the subspace. For an ellipsoidal trust region, however, the allowed step lengths also play a role, as the same subspace may yield different decreases depending on the shape matrix.  Inspired by the notion of $\alpha$-well-aligned matrices \cite{cartis2023scalable,chen2024qfully,dzahini2024stochastic}, we combine both effects in the following definition.
    \begin{definition}\label{def:ellipsubalignment}
        Let $f\in\mathcal{C}^1$ and $\boldsymbol{x}\in\mathbb{R}^n$.  Let $\boldsymbol{Q}\in\mathbb{R}^{n\times p}$ have orthonormal columns and $\boldsymbol{S}\succ 0$.  We define the ellipsoidal subspace alignment of $f$ at $\boldsymbol{x}$ by
        \begin{equation*}
            \Xi(\boldsymbol{x};\boldsymbol{Q};\boldsymbol{S}):=\begin{cases}
                \frac{\nabla f(\boldsymbol{x})^\top\boldsymbol{Q}
               (\boldsymbol{Q}^\top\boldsymbol{S}\boldsymbol{Q})^{-1}
               \boldsymbol{Q}^\top\nabla f(\boldsymbol{x})}
               {\|\nabla f(\boldsymbol{x})\|^2}, &\text{if $\nabla f(\boldsymbol{x})\neq\boldsymbol{0}_n$},\\
               0, &\text{if $\nabla f(\boldsymbol{x})=\boldsymbol{0}_n$}.
            \end{cases}
        \end{equation*}
    \end{definition}
    In this expression, $\boldsymbol{Q}^\top\nabla f(\boldsymbol{x})$ measures the gradient retained by the subspace, while $(\boldsymbol{Q}^\top\boldsymbol{S}\boldsymbol{Q})^{-1}$ weights its components according to the step lengths permitted by the ellipsoidal section. More precisely, for a nonzero gradient, $\Delta\sqrt{\Xi(\boldsymbol{x};\boldsymbol{Q};\boldsymbol{S})}\|\nabla f(\boldsymbol{x})\|$ is the largest decrease in the linear approximation of $f$ at $\boldsymbol{x}$ over that section, so it accounts for both the available descent directions and the permitted step lengths.
    \begin{remark}
        If $\boldsymbol{S}=\boldsymbol{I}_n$ and $\nabla f(\boldsymbol{x})\ne\boldsymbol{0}_n$, then $\Xi(\boldsymbol{x};\boldsymbol{Q};\boldsymbol{S})$ equals $\|\boldsymbol{Q}^\top\nabla f(\boldsymbol{x})\|^2\slash\|\nabla f(\boldsymbol{x})\|^2$, which is the squared projected-gradient fraction used in the definition of $\alpha$-well-alignment in \cite{cartis2023scalable,chen2024qfully,dzahini2024stochastic}.  The value of $\Xi(\boldsymbol{x};\boldsymbol{Q};\boldsymbol{S})$ is unchanged by replacing $\boldsymbol Q$ by another orthonormal basis of the same subspace.  Moreover, $\Xi(\boldsymbol{x};\boldsymbol{Q};\boldsymbol{S})$ need not lie in $[0,1]$; under Inequality~\eqref{ineq:shape-bounds-Sk} it lies in $[0,1/\underline\lambda_S]$.
    \end{remark}

    The upper bound on the shape matrix ensures that a subspace retaining a sufficiently large fraction of the full gradient also has sufficiently large ellipsoidal subspace alignment.  Indeed, under Inequality~\eqref{ineq:shape-bounds-Sk}, whenever $\nabla f(\boldsymbol{x}_k)\ne\boldsymbol{0}_n$,
    \begin{equation}\label{eq:Xi-lower}
        \Xi(\boldsymbol{x}_k;\boldsymbol{Q}_k;\boldsymbol{S}_k)\ge\frac{\|\boldsymbol{Q}_k^\top\nabla f(\boldsymbol{x}_k)\|^2}{\overline\lambda_S\|\nabla f(\boldsymbol{x}_k)\|^2}.
    \end{equation}

    For simplicity, we denote $\Xi_k:=\Xi(\boldsymbol{x}_k;\boldsymbol{Q}_k;\boldsymbol{S}_k)$.  As we will see in Section \ref{sec:convergcomplexity}, we need to sample subspaces such that $\Xi_k$ exceeds a positive threshold with a uniformly positive conditional probability, as stated in Assumption~\ref{ass:mean-xi}.  In Subsection~\ref{subsec:subspacesampling}, we discuss two sampling methods that satisfy this assumption.

\subsection{Algorithm}
    RETRO, stated in Algorithm~\ref{alg:main}, maintains the current point $\boldsymbol{x}_k$, the shape matrix $\boldsymbol{S}_k$, and the scalar size parameter $\Delta_k$.  At the start of an iteration, we sample a subspace with orthonormal basis $\boldsymbol{Q}_k$ and construct a model $\widehat m_k$ of the objective restricted to that subspace.  Each model is a polynomial of degree at most two, so its Hessian is a fixed matrix.  The model is required to be accurate on the Euclidean ball containing the current ellipsoidal section.  Its accuracy is therefore related to the reduced outer radius $\widehat r_k^{\mathrm{out}}$.  To account for the shape of the section, set
    \begin{equation}\label{eq:Pk-definition}
        \boldsymbol P_k:=(\boldsymbol Q_k^\top\boldsymbol S_k\boldsymbol Q_k)^{-1/2}~~~\text{and so}~~~\widehat r_k^{\mathrm{out}}=\Delta_k\|\boldsymbol P_k\|.
    \end{equation}
    Under the change of variables $\widehat{\boldsymbol{s}}=\boldsymbol{P}_k\widehat{\boldsymbol{y}}$, the reduced ellipsoidal section becomes a ball of radius $\Delta_k$, and the model gradient at its center becomes $\boldsymbol{P}_k\nabla\widehat m_k(\boldsymbol{0}_p)$. This transformation is described in Subsection~\ref{subsec:trsubproblem}.  The global contraction and expansion actions used below will be defined in Subsection~\ref{subsec:ellipsoidupdate}; they combine a change in the overall size with an additional change along the sampled subspace. The criticality test compares the transformed model-gradient norm with $\widehat r_k^{\mathrm{out}}$. If the model gradient is too small relative to this radius, the algorithm keeps the current point and applies the global contraction action.  We call this a criticality iteration.  

    Otherwise, we approximately minimize the model over the ellipsoidal section.  The reduced step $\widehat{\boldsymbol{s}}_k$ gives the full-space trial step $\boldsymbol{s}_k=\boldsymbol{Q}_k\widehat{\boldsymbol{s}}_k$.  The ratio $\rho_k$ compares the actual decrease in the objective with the decrease predicted by the model.  If $\rho_k<\eta_1$, the trial step is rejected and the global contraction action is applied.  If $\eta_1\le\rho_k<\eta_2$, the step is accepted and the ellipsoid keeps its size and shape.  If $\rho_k\ge\eta_2$, the step is accepted and the global expansion action is applied.  The accepted iterations are called successful, and those in the last case are called very successful.

    The contraction and expansion actions adjust both $\Delta_k$ and $\boldsymbol{S}_k$, with an additional change along the sampled subspace.  Their formulas are given in Subsection~\ref{subsec:ellipsoidupdate}.  After the update, eigenvalue clipping enforces Inequality~\eqref{ineq:shape-bounds-Sk} and a cap limits the full-space outer radius.  The next trust region is centered at $\boldsymbol{x}_{k+1}$.  The algorithm stops when its outer radius falls below $\mathrm{radi}_{\min}$.
    \begin{algorithm}
    \caption{RETRO: Random-subspace Ellipsoidal TRust-region Optimization}
    \label{alg:main}
    \begin{algorithmic}[1]
    \REQUIRE $f$, $\boldsymbol{x}_0$, $1\le p\le n$, $\Delta_0>0$,
    $0<\underline\lambda_S\le\overline\lambda_S$, $\boldsymbol{S}_0\succ0$ satisfying Inequality~\eqref{ineq:shape-bounds-Sk}, $0<\gamma_{\mathrm{dec}}^{\mathrm{sub}}<1$,
    $0<\gamma_{\mathrm{dec}}^{\mathrm{ful}}<1$,
    $\gamma_{\mathrm{inc}}^{\mathrm{sub}}>1$,
    $\gamma_{\mathrm{inc}}^{\mathrm{ful}}>1$,
    $0<\eta_1\le\eta_2<1$, $\mu>0$, and
    $0\le\mathrm{radi}_{\min}\le r_0^{\mathrm{out}}\le\mathrm{radi}_{\max}$.
    \FOR{$k=0,1,2,\ldots$}
        \STATE Select $\boldsymbol{Q}_k\in\mathbb R^{n\times p}$ with orthonormal columns, and compute $\boldsymbol P_k$ and $\widehat r_k^{\mathrm{out}}$ from Equations~\eqref{eq:Pk-definition}.
        \STATE Construct $\widehat m_k$ that belongs to a class of $\boldsymbol{Q}_k$-fully linear models of $f$ at $\boldsymbol{x}_k$ parameterized by $\widehat{r}_k^{\mathrm{out}}$.
        \IF{$\mu \|\boldsymbol{P}_k\nabla\widehat m_k(\boldsymbol{0}_p)\|<\widehat{r}_k^{\mathrm{out}}$}
            \STATE Set $\boldsymbol{x}_{k+1}=\boldsymbol{x}_k$ and apply the
            \emph{global contraction} action.
        \ELSE
            \STATE Compute a feasible $\widehat{\boldsymbol{s}}_k$ for Problem~\eqref{pro:subspaceTR} satisfying
            Inequality~\eqref{ineq:Cauchydec_m}, set
            $\boldsymbol{s}_k=\boldsymbol{Q}_k\widehat{\boldsymbol{s}}_k$, and compute
            \begin{equation*}
                \rho_k:=
                \frac{f(\boldsymbol{x}_k)-f(\boldsymbol{x}_k+\boldsymbol{s}_k)}
                {\widehat m_k(\boldsymbol{0}_p)
                 -\widehat m_k(\widehat{\boldsymbol{s}}_k)}.
            \end{equation*}
            \IF{$\rho_k<\eta_1$}
                \STATE Set $\boldsymbol{x}_{k+1}=\boldsymbol{x}_k$ and apply the
                \emph{global contraction} action.
            \ELSIF{$\rho_k<\eta_2$}
                \STATE Set $\boldsymbol{x}_{k+1}=\boldsymbol{x}_k+\boldsymbol{s}_k$
                and apply the \emph{keep} action.
            \ELSE
                \STATE Set $\boldsymbol{x}_{k+1}=\boldsymbol{x}_k+\boldsymbol{s}_k$
                and apply the \emph{global expansion} action.
            \ENDIF
        \ENDIF
        \STATE Set $\boldsymbol{S}_{k+1}$ and $\Delta_{k+1}$ according to Subsection~\ref{subsec:ellipsoidupdate}.
        \IF{$r^{\mathrm{out}}_{k+1}<\mathrm{radi}_{\min}$}
            \STATE \textbf{stop}.
        \ENDIF
    \ENDFOR
    \end{algorithmic}
    \end{algorithm}
    \FloatBarrier

\subsection{Model construction}
    The model must approximate objective values well enough for the acceptance ratio and projected gradients well enough for the criticality test.  Following \cite[Definition 1]{cartis2023scalable} and \cite[Definition 6]{chen2024qfully}, we introduce the notion of $\boldsymbol{Q}$-fully linear models.
    \begin{definition}
        Given $f\in\mathcal{C}^1$, $\boldsymbol{x}\in\mathbb{R}^n$,
        $\overline\Delta>0$, and $\boldsymbol{Q}\in\mathbb{R}^{n\times p}$,
        we say that $\mathcal{M}_{\overline\Delta}
        =\{\widehat m_\Delta:\mathbb{R}^p\to\mathbb{R}\}_{\Delta\in(0,\overline\Delta]}$
        is a class of $\boldsymbol{Q}$-fully linear models of $f$ at
        $\boldsymbol{x}$ parameterized by $\Delta$ if there exist constants
        $\kappa_{ef}(\boldsymbol{x}),\kappa_{eg}(\boldsymbol{x})>0$ such that,
        for all $\Delta\in(0,\overline\Delta]$ and $\|\widehat{\boldsymbol{s}}\|\le\Delta$,
        \begin{align*}
            \left|f(\boldsymbol{x}+\boldsymbol{Q}\widehat{\boldsymbol{s}})
                  -\widehat m_\Delta(\widehat{\boldsymbol{s}})\right|
            &\le\kappa_{ef}(\boldsymbol{x})\Delta^2,\\
            \left\|\boldsymbol{Q}^\top\nabla f(\boldsymbol{x}+\boldsymbol{Q}\widehat{\boldsymbol{s}})
                  -\nabla\widehat m_\Delta(\widehat{\boldsymbol{s}})\right\|
            &\le\kappa_{eg}(\boldsymbol{x})\Delta.
        \end{align*}
    \end{definition}

    Here, ``fully linear'' describes the approximation accuracy, rather than the degree of the model.  In particular, a quadratic model can satisfy this definition.  The constants may depend on the center and the subspace, whereas Assumption~\ref{ass:fullylinear} below requires uniform constants along the algorithm.  The following three constructions from \cite[Section 4.1]{chen2024qfully} provide this accuracy using only function values.

    Write $\widehat f(\widehat{\boldsymbol s}):=f(\boldsymbol x+\boldsymbol Q\widehat{\boldsymbol s})$.  Consider
    \begin{enumerate}
        \item the linear model interpolating $\widehat f$ at $\boldsymbol0_p$ and $\Delta\boldsymbol e^i$, for all $1\le i\le p$;
        \item the quadratic model with diagonal Hessian interpolating $\widehat f$ at $\boldsymbol0_p$, $\Delta\boldsymbol e^i$, and $2\Delta\boldsymbol e^i$,  for all $1\le i\le p$;
        \item the (determined) quadratic model interpolating $\widehat f$ at $\boldsymbol0_p$, $\Delta\boldsymbol e^i$, and $2\Delta\boldsymbol e^i$ for $1\le i\le p$, and at $\Delta(\boldsymbol e^i+\boldsymbol e^j)$ for $1\le i<j\le p$.
    \end{enumerate}
    All three models are uniquely specified by these conditions.  The linear model uses $p+1$ points.  The diagonal quadratic model adds one point along each coordinate direction to estimate curvature, but omits mixed quadratic terms; it is underdetermined among general quadratic polynomials when $p>1$.  The determined quadratic model includes those mixed terms and uses $(p+1)(p+2)/2$ points.  The next proposition shows that all three choices satisfy the required accuracy and Hessian bounds with constants that do not change as the center, subspace, or radius changes.

    \begin{proposition}\label{prop:orthogonal-models}
        Suppose that $f\in\mathcal{C}^{1+}$ with Lipschitz constant $L_{\nabla f}>0$ and $\boldsymbol Q$ has orthonormal columns. Each of the three models above belongs to a class of $\boldsymbol Q$-fully linear models of $f$ at $\boldsymbol x$ parameterized by $\Delta$, with the common bounds
        \begin{equation*}
            \kappa_{ef}=L_{\nabla f}\left(\frac{p+1}{2}+\sqrt p\right)~~~\text{and}~~~\kappa_{eg}=L_{\nabla f}(p+1+\sqrt p).
        \end{equation*}
        Their Hessians satisfy $\|\nabla^2\widehat m\|\le pL_{\nabla f}$. For the linear and diagonal quadratic models this last bound can be replaced by $0$ and $L_{\nabla f}$, respectively. All these bounds are independent of $\Delta$, $\boldsymbol x$, and $\boldsymbol Q$.
    \end{proposition}
    \begin{proof}
        As in \cite[Lemma~3]{chen2024qfully}, we can show that $\nabla\widehat f$ is Lipschitz continuous with constant $L_{\nabla f}$. From, e.g., \cite[Lemma~9.4]{audet2017derivative}, we have
        \begin{equation}\label{eq:model-taylor}
            |\widehat f(\widehat{\boldsymbol s})-\widehat f(\boldsymbol0_p)
             -\nabla\widehat f(\boldsymbol0_p)^\top\widehat{\boldsymbol s}|
            \le\frac12L_{\nabla f}\|\widehat{\boldsymbol s}\|^2.
        \end{equation}
        For either quadratic model, interpolation on each coordinate line gives
        \begin{align*}
            (\nabla^2\widehat m)_{ii}
            &=\frac{\widehat f(2\Delta\boldsymbol e^i)-2\widehat f(\Delta\boldsymbol e^i)+\widehat f(\boldsymbol0_p)}{\Delta^2},\\
            (\nabla\widehat m(\boldsymbol0_p))_i
            &=\frac{\widehat f(\Delta\boldsymbol e^i)-\widehat f(\boldsymbol0_p)}\Delta
              -\frac \Delta2(\nabla^2\widehat m)_{ii}.
        \end{align*}
        The numerator in the first equation equals $\int_0^\Delta\boldsymbol e^{i\top}(\nabla\widehat f((\Delta+t)\boldsymbol e^i)-\nabla\widehat f(t\boldsymbol e^i))\mathrm{d}t$.  Consequently $|(\nabla^2\widehat m)_{ii}|\le L_{\nabla f}$, and Inequality~\eqref{eq:model-taylor} gives
        \begin{equation*}
            \left|(\nabla\widehat m(\boldsymbol0_p)-\nabla\widehat f(\boldsymbol0_p))_i\right| \le \frac{1}{\Delta}\left|\widehat f(\Delta\boldsymbol e^i)-\widehat f(\boldsymbol0_p)-\nabla\widehat f(\boldsymbol0_p)^\top(\Delta\boldsymbol e^i)\right| + \frac{L_{\nabla f}}{2}\Delta \le L_{\nabla f}\Delta.
        \end{equation*}
        For the determined quadratic model and $i\ne j$, the mixed entry is
        \begin{equation*}
            (\nabla^2\widehat m)_{ij}
            =\frac{\widehat f(\Delta(\boldsymbol e^i+\boldsymbol e^j))
                    -\widehat f(\Delta\boldsymbol e^i)-\widehat f(\Delta\boldsymbol e^j)
                    +\widehat f(\boldsymbol0_p)}{\Delta^2}.
        \end{equation*}
        Its numerator is
        $\int_0^\Delta\boldsymbol e^{i\top}(\nabla\widehat f(t\boldsymbol e^i+\Delta\boldsymbol e^j)-\nabla\widehat f(t\boldsymbol e^i))\mathrm{d}t$,
        so $|(\nabla^2\widehat m)_{ij}|\le L_{\nabla f}$. The Hessian bound follows from
        \begin{equation*}
            \left\|\nabla^2\widehat m\right\| \le \left\|\nabla^2\widehat m\right\|_F = \sqrt{\sum_{i,j=1}^p(\nabla^2\widehat m)_{ij}^2} \le \sqrt{p^2L_{\nabla f}^2} = pL_{\nabla f}.
        \end{equation*}
        For the linear model the Hessian is zero and 
        \begin{equation*}
            \left|(\nabla\widehat m(\boldsymbol0_p)-\nabla\widehat f(\boldsymbol0_p))_i\right| = \frac{1}{\Delta}\left|\widehat f(\Delta\boldsymbol e^i)-\widehat f(\boldsymbol0_p)-\nabla\widehat f(\boldsymbol0_p)^\top(\Delta\boldsymbol e^i)\right|\le \frac{L_{\nabla f}}{2}\Delta.
        \end{equation*}
        Thus, all three models satisfy
        \begin{equation*}
            \|\nabla\widehat m(\boldsymbol0_p)-\nabla\widehat f(\boldsymbol0_p)\| \le\sqrt p L_{\nabla f}\Delta~~~\text{and}~~~\|\nabla^2\widehat m\|\le pL_{\nabla f}.
        \end{equation*}
        For $\|\widehat{\boldsymbol s}\|\le \Delta$, the claimed constants follow from combining the above bounds, Inequality~\eqref{eq:model-taylor}, the Lipschitz continuity of $\nabla\widehat f$,
        \begin{align*}
            &~~\left|f(\boldsymbol{x}+\boldsymbol{Q}\widehat{\boldsymbol s})-\widehat m(\widehat{\boldsymbol s})\right|\\
            &= \left|\widehat f(\widehat{\boldsymbol s})-\widehat m(\boldsymbol0_p)-\nabla\widehat m(\boldsymbol0_p)^\top\widehat{\boldsymbol s}-\frac{1}{2}\widehat{\boldsymbol s}^\top\nabla^2\widehat m\widehat{\boldsymbol s}\right|\\
            &\le \left|\widehat f(\widehat{\boldsymbol s})-\widehat f(\boldsymbol0_p)
             -\nabla\widehat f(\boldsymbol0_p)^\top\widehat{\boldsymbol s}\right| + \left\|\nabla\widehat f(\boldsymbol0_p)-\nabla\widehat m(\boldsymbol0_p)\right\|\Delta + \frac{1}{2}\left\|\nabla^2\widehat m\right\|\Delta^2,
        \end{align*}
        and 
        \begin{align*}
            &~~\left\|\boldsymbol{Q}^\top\nabla f(\boldsymbol{x}+\boldsymbol{Q}\widehat{\boldsymbol{s}}) -\nabla\widehat m(\widehat{\boldsymbol{s}})\right\|\\
            &= \left\|\nabla\widehat f(\widehat{\boldsymbol s})-\nabla\widehat m(\boldsymbol0_p)-\nabla^2\widehat m\widehat{\boldsymbol{s}}\right\|\\
            &\le \left\|\nabla\widehat f(\widehat{\boldsymbol s})-\nabla\widehat f(\boldsymbol0_p)\right\| + \left\|\nabla\widehat f(\boldsymbol0_p)-\nabla\widehat m(\boldsymbol0_p)\right\| + \left\|\nabla^2\widehat m\right\|\Delta.
        \end{align*}
    \end{proof}

    In RETRO, take $\Delta=\widehat r_k^{\mathrm{out}}$. Proposition~\ref{prop:orthogonal-models} directly verifies Assumptions~\ref{ass:fullylinear} and \ref{ass:bdmodelHess} below. We note that if $f\in\mathcal C^{2+}$, the determined quadratic model also has the stronger fully quadratic accuracy established in \cite[Theorems 5 and 6]{chen2024qfully}. That stronger property is not needed in the analysis of this paper.

    Other quadratic model constructions may also be used, provided that the corresponding accuracy and Hessian bounds can be established. In particular, \cite{chen2026accuracy} derives fully linear bounds for minimum norm, minimum Frobenius norm, and quadratic generalized simplex derivative models. Combined with the equivalence results between full-space and subspace models established in \cite{chen2026relationships}, these bounds can be used to verify Assumptions~\ref{ass:fullylinear} and \ref{ass:bdmodelHess} for the corresponding model constructions.

\subsection{Trust-region subproblem}\label{subsec:trsubproblem}
    Once the criticality test is passed, we seek a decrease in the model within the current ellipsoidal section.  This requires solving a problem in $p$ variables.  The reduced trust-region subproblem is
    \begin{equation}\label{pro:subspaceTR}
        \min_{\widehat{\boldsymbol{s}}\in\mathbb R^p}
        \widehat m_k(\widehat{\boldsymbol{s}})~~~\text{s.t.}~~~\widehat{\boldsymbol{s}}^\top\boldsymbol{Q}_k^\top\boldsymbol{S}_k\boldsymbol{Q}_k
        \widehat{\boldsymbol{s}}\le\Delta_k^2.
    \end{equation}
    Using $\boldsymbol P_k$ from Equations~\eqref{eq:Pk-definition}, set $\widehat{\boldsymbol{y}}:=\boldsymbol{P}_k^{-1}\widehat{\boldsymbol{s}}$.  Then, Problem~\eqref{pro:subspaceTR} is equivalent to minimizing $\widehat M_k(\widehat{\boldsymbol{y}})
    :=\widehat m_k(\boldsymbol{P}_k\widehat{\boldsymbol{y}})$ subject to $\widehat{\boldsymbol y}\in \mathcal{B}(\boldsymbol{0}_p;\Delta_k)$.  This change of coordinates allows us to use the usual step computation for a ball trust region.  The standard Cauchy decrease requirement \cite{conn2000trust,conn2009introduction} on $\widehat M_k(\widehat{\boldsymbol{y}})$ is
    \begin{equation}\label{ineq:Cauchydec_M}
        \widehat M_k(\boldsymbol{0}_p)-\widehat M_k(\widehat{\boldsymbol{y}}_k) \ge c_{\mathrm{tr}}\left\|\nabla\widehat M_k(\boldsymbol{0}_p)\right\|\min\left\{\Delta_k, \frac{\|\nabla\widehat M_k(\boldsymbol{0}_p)\|}{\|\nabla^2\widehat M_k(\boldsymbol{0}_p)\|+1}\right\},
    \end{equation}
    for some $c_{\mathrm{tr}}\in(0,1/2]$ independent of $k$.  Define $\widehat{\boldsymbol{s}}_k:=\boldsymbol{P}_k\widehat{\boldsymbol{y}}_k$.  Then, Inequality \eqref{ineq:Cauchydec_M} transforms into
    \begin{equation}\label{ineq:Cauchydec_m}
        \widehat m_k(\boldsymbol{0}_p)-\widehat m_k(\widehat{\boldsymbol{s}}_k)
        \ge c_{\mathrm{tr}}\left\|\boldsymbol{P}_k\nabla\widehat m_k(\boldsymbol{0}_p)\right\|\min\left\{\Delta_k,\frac{\|\boldsymbol{P}_k\nabla\widehat m_k(\boldsymbol{0}_p)\|}{\|\boldsymbol{P}_k\nabla^2\widehat m_k(\boldsymbol{0}_p)\boldsymbol{P}_k\|+1}\right\}.
    \end{equation}
    This is the quality we need on the solution to the trust-region subproblem.

\subsection{Ellipsoid update}\label{subsec:ellipsoidupdate}
    The ellipsoid update changes the overall size through $\Delta_k$ and adjusts the shape along the sampled subspace through $\boldsymbol{S}_k$.  This allows the next region to reflect the outcome of the current step.  The full-space factor changes all directional radii equally, while the subspace factor makes an additional change along $\mathrm{col}(\boldsymbol{Q}_k)$.  For $\gamma>0$, define $\boldsymbol{M}_k(\gamma):=\boldsymbol{I}_n+(\gamma^{-1}-1)\boldsymbol{Q}_k\boldsymbol{Q}_k^\top$. Given a subspace factor $\gamma^{\mathrm{sub}}>0$ and a full-space factor $\gamma^{\mathrm{ful}}>0$, we let
    \begin{equation*}
        \widetilde{\boldsymbol{S}}_{k+1}:=\boldsymbol{M}_k(\gamma^{\mathrm{sub}})\boldsymbol{S}_k\boldsymbol{M}_k(\gamma^{\mathrm{sub}})~~~\text{and}~~~\widetilde\Delta_{k+1}:=\gamma^{\mathrm{ful}}\Delta_k.
    \end{equation*}
    Then, for any $\boldsymbol{x}\in\mathbb{R}^n$, the directional radii satisfy
    \begin{equation*}
        \mathrm{radi}((\mathcal T(\boldsymbol{x};\widetilde{\boldsymbol{S}}_{k+1};\widetilde\Delta_{k+1})-\boldsymbol{x})
           \cap\mathrm{span}(\boldsymbol{u})) =\gamma^{\mathrm{sub}}\gamma^{\mathrm{ful}}
          \mathrm{radi}((\mathcal T(\boldsymbol{x};\boldsymbol{S}_k;\Delta_k)-\boldsymbol{x})
           \cap\mathrm{span}(\boldsymbol{u}))
    \end{equation*}
    for all unit vectors $\boldsymbol{u}\in\mathrm{col}(\boldsymbol{Q}_k)$, and
    \begin{equation*}
        \mathrm{radi}((\mathcal T(\boldsymbol{x};\widetilde{\boldsymbol{S}}_{k+1};\widetilde\Delta_{k+1})-\boldsymbol{x})\cap\mathrm{span}(\boldsymbol{u})) =\gamma^{\mathrm{ful}}
          \mathrm{radi}((\mathcal T(\boldsymbol{x};\boldsymbol{S}_k;\Delta_k)-\boldsymbol{x})\cap\mathrm{span}(\boldsymbol{u}))
    \end{equation*}
    for all unit vectors $\boldsymbol{u}\in\mathrm{col}(\boldsymbol{Q}_k)^\perp$.  Indeed, notice that we have $\boldsymbol{M}_k(\gamma)\boldsymbol{u}=\gamma^{-1}\boldsymbol{u}$ on $\mathrm{col}(\boldsymbol{Q}_k)$ and $\boldsymbol{M}_k(\gamma)\boldsymbol{u}=\boldsymbol{u}$ on $\mathrm{col}(\boldsymbol{Q}_k)^\perp$. Equivalently, 
    \begin{equation*}
        \mathcal T(\boldsymbol{x};\widetilde{\boldsymbol{S}}_{k+1};\widetilde\Delta_{k+1})
        =\left\{\boldsymbol{x}+\gamma^{\mathrm{ful}}\boldsymbol{M}_k(\gamma^{\mathrm{sub}})^{-1}\left(\boldsymbol{s}-\boldsymbol{x}\right):\boldsymbol{s}\in\mathcal T(\boldsymbol{x};\boldsymbol{S}_k;\Delta_k)\right\}.
    \end{equation*}
    When $\gamma^{\mathrm{sub}}\le1$, the linear part of this affine map has operator norm at most $\gamma^{\mathrm{ful}}$ and $$\mathrm{radi}(\mathcal T(\boldsymbol{x};\widetilde{\boldsymbol{S}}_{k+1};\widetilde\Delta_{k+1})-\boldsymbol{x})
    \le\gamma^{\mathrm{ful}}r_k^{\mathrm{out}};$$ it is a strict contraction when $\gamma^{\mathrm{ful}}<1$.  
    
    To enforce Inequality~\eqref{ineq:shape-bounds-Sk}, let
    $\boldsymbol{S}_{k+1}$ be obtained by clipping the eigenvalues of $\widetilde{\boldsymbol{S}}_{k+1}$ to
    $[\underline\lambda_S,\overline\lambda_S]$.  The scalar parameter is then set to
    \begin{equation}\label{eq:outer-cap}
        \Delta_{k+1}
        :=\min\left\{
           \widetilde\Delta_{k+1},
           \mathrm{radi}_{\max}
           \sqrt{\lambda_{\min}(\boldsymbol{S}_{k+1})}
           \right\}.
    \end{equation}
    
    Clipping limits the shape of the ellipsoid, and Equation~\eqref{eq:outer-cap} ensures that its full-space outer radius does not exceed $\mathrm{radi}_{\max}$.  The three actions in RETRO use the factors listed in Table~\ref{tab:actions}, where $0<\gamma_{\mathrm{dec}}^{\mathrm{sub}}<1$,
    $0<\gamma_{\mathrm{dec}}^{\mathrm{ful}}<1$,
    $\gamma_{\mathrm{inc}}^{\mathrm{sub}}>1$, and
    $\gamma_{\mathrm{inc}}^{\mathrm{ful}}>1$.
    \begin{table}\centering 
    \caption{Subspace and full-space factors for the ellipsoid update.} 
    \label{tab:actions}
    \begin{tabular}{c|cc}
    \toprule
     Action & $\gamma^{\mathrm{sub}}$ & $\gamma^{\mathrm{ful}}$\\ 
     \midrule
     global contraction
       & $\gamma_{\mathrm{dec}}^{\mathrm{sub}}$
       & $\gamma_{\mathrm{dec}}^{\mathrm{ful}}$\\
     keep
       & $1$ & $1$\\
     global expansion
       & $\gamma_{\mathrm{inc}}^{\mathrm{sub}}$
       & $\gamma_{\mathrm{inc}}^{\mathrm{ful}}$\\ 
       \bottomrule 
    \end{tabular}
    \end{table} 
    
    Under these factors, for the global contraction action, we have $\Delta_{k+1}=\gamma_{\mathrm{dec}}^{\mathrm{ful}}\Delta_k$.
    Indeed, notice that $\boldsymbol{M}_k(\gamma_{\mathrm{dec}}^{\mathrm{sub}})\succeq\boldsymbol{I}_n$ gives
    \begin{equation*}
        \lambda_{\min}(\widetilde{\boldsymbol{S}}_{k+1})
        \ge \lambda_{\min}(\boldsymbol{S}_k)
        \lambda_{\min}(\boldsymbol{M}_k(\gamma_{\mathrm{dec}}^{\mathrm{sub}}))^2
        \ge \lambda_{\min}(\boldsymbol{S}_k).
    \end{equation*}
    Since $\lambda_{\min}(\boldsymbol{S}_k)\le\overline\lambda_S$, clipping also gives $\lambda_{\min}(\boldsymbol{S}_{k+1})\ge\lambda_{\min}(\boldsymbol{S}_k)$.
    Together with $\Delta_k\le\mathrm{radi}_{\max}\sqrt{\lambda_{\min}(\boldsymbol{S}_k)}$ and $\gamma_{\mathrm{dec}}^{\mathrm{ful}}<1$, this implies $\widetilde\Delta_{k+1}\le\mathrm{radi}_{\max}\sqrt{\lambda_{\min}(\boldsymbol{S}_{k+1})}$, so the cap in Equation~\eqref{eq:outer-cap} is inactive.
    For the keep action, we have $\Delta_{k+1}=\Delta_k$.  For the global expansion
    action, we have $\Delta_{k+1}\le\gamma_{\mathrm{inc}}^{\mathrm{ful}}\Delta_k$ and equality holds whenever $\gamma_{\mathrm{inc}}^{\mathrm{ful}}\Delta_k\le\mathrm{radi}_{\max}\sqrt{\underline\lambda_S}$.

\section{Convergence and complexity analysis}\label{sec:convergcomplexity}
    In this section, we analyze the convergence and complexity of RETRO.  We first show that a successful iteration yields a decrease that depends on the subspace alignment and the current size parameter.  Then, we show that when the size parameter is sufficiently small and the subspace is sufficiently well aligned, the iteration is very successful and the size parameter is increased.  We combine these results with the update rules to bound the number of iterations required for the gradient norm to become small. Under an appropriate subspace sampling assumption, we then establish an expected complexity bound and almost-sure first-order liminf convergence. Finally, we discuss this sampling assumption and verify it for two subspace sampling strategies: Haar and metric-guided sampling.

    The subspace may be sampled using information from earlier iterations.  We therefore state its probability of sufficient alignment conditional on the information already available.  Let $\mathcal F_{k-1}$ denote the information available immediately before $\boldsymbol{Q}_k$ is sampled; in particular, $\boldsymbol{x}_k$, $\Delta_k$, and $\boldsymbol{S}_k$ are $\mathcal F_{k-1}$-measurable.  Let $\mathcal F_k$ contain all random quantities generated through the end of iteration $k$, so that $\mathcal F_{k-1}\subseteq\mathcal F_k$.

    We make use of the following assumptions.
    \begin{assumption}\label{ass:smoothness-bdbelow}
        The function $f\in\mathcal C^{1+}$ has Lipschitz continuous gradient with constant $L_{\nabla f}>0$ and is bounded below by $f_{\mathrm{low}}\in\mathbb R$.
    \end{assumption}
    \begin{assumption}\label{ass:fullylinear}
        For every $k$, the model $\widehat m_k$ belongs to a class of $\boldsymbol{Q}_k$-fully linear models of $f$ at $\boldsymbol{x}_k$ parameterized by $\widehat{r}_k^{\mathrm{out}}$ with constants $\kappa_{ef}$ and $\kappa_{eg}$ independent of $k$.
    \end{assumption}
    \begin{assumption}\label{ass:bdmodelHess}
        There exists $\kappa_H\ge0$ such that $\|\nabla^2\widehat m_k\|\le\kappa_H$ for all $k$.
    \end{assumption}
    \begin{assumption}\label{ass:CauchyDec}
        Every computed step $\widehat{\boldsymbol{s}}_k$ satisfies the transformed Cauchy decrease \eqref{ineq:Cauchydec_m}.
    \end{assumption}
    \begin{assumption}\label{ass:mean-xi}
        For a given $\epsilon>0$, there exist constants $\overline\tau_\epsilon\in(0,1/\underline\lambda_S]$ and
        $\pi_\epsilon\in(0,1]$, independent of $k$, such that whenever
        $\|\nabla f(\boldsymbol{x}_k)\|>\epsilon$, $\mathbb P(\Xi_k\ge\overline\tau_\epsilon
               \mid\mathcal F_{k-1})\ge\pi_\epsilon$.
    \end{assumption}
    Assumptions~\ref{ass:smoothness-bdbelow} to \ref{ass:bdmodelHess} are standard in the analysis of random-subspace model-based trust-region methods \cite{cartis2023scalable,chen2024qfully,dzahini2024stochastic}, with model accuracy specified here using the outer radius of the ellipsoidal section. Assumption~\ref{ass:CauchyDec} is the standard Cauchy decrease requirement \cite{conn2000trust,conn2009introduction} expressed in the transformed coordinates described in Subsection~\ref{subsec:trsubproblem}, which can be satisfied by standard trust-region subproblem solvers that guarantee a fixed fraction of Cauchy decrease \cite[Chapter~7]{conn2000trust}. Assumption~\ref{ass:mean-xi} requires the sampling rule to produce a sufficiently aligned subspace with a conditional probability bounded below uniformly over the iterations whenever the full gradient norm exceeds the target tolerance. This adapts the probabilistic $\alpha$-well-alignment conditions in \cite{cartis2023scalable,chen2024qfully,dzahini2024stochastic} to the ellipsoidal setting and allows both the alignment threshold and the probability bound to depend on the target tolerance.

    \begin{remark}\label{rem:ordinary-alignment}
        Under Inequality~\eqref{ineq:shape-bounds-Sk}, the probabilistic $\alpha$-well-alignment condition used in \cite{cartis2023scalable,chen2024qfully,dzahini2024stochastic} is sufficient for Assumption~\ref{ass:mean-xi}. Indeed, suppose that, for constants
        $\alpha,\delta_S\in(0,1)$,
        \begin{equation}\label{eq:ordinary-alignment}
            \mathbb P\left(
                \|\boldsymbol{Q}_k^\top\nabla f(\boldsymbol{x}_k)\|
                \ge\alpha\|\nabla f(\boldsymbol{x}_k)\|
                \ \middle|\ \mathcal F_{k-1}\right)\ge1-\delta_S.
        \end{equation}
        Then, Inequality~\eqref{eq:Xi-lower} gives Assumption~\ref{ass:mean-xi} with $\overline\tau_\epsilon=\alpha^2/\overline\lambda_S$ and
        $\pi_\epsilon=1-\delta_S$.
    \end{remark}

\subsection{Successful iterations}
    We first relate the outcome of one successful iteration to the gradient norm, the alignment, and the size parameter.  Suppose that we have run RETRO for $N>0$ iterations and the stopping condition is not triggered.  We define the following subsets of $\mathcal{K}_N:=\{0,\ldots,N-1\}$:
    \begin{itemize}
        \item $\mathcal{C}_N$: the set of all criticality iterations in $\mathcal{K}_N$, i.e., the set of all iterations in $\mathcal{K}_N$ with $\mu \|\boldsymbol{P}_k\nabla\widehat m_k(\boldsymbol{0}_p)\|<\widehat{r}_k^{\mathrm{out}}$.
        \item $\mathcal{S}_N$: the set of all successful iterations in $\mathcal{K}_N$, i.e., the set of all iterations that pass the criticality test in $\mathcal{K}_N$ with $\rho_k\ge\eta_1$.
        \item $\mathcal{V}_N$: the set of all very successful iterations in $\mathcal{K}_N$, i.e., the set of all iterations that pass the criticality test in $\mathcal{K}_N$ with $\rho_k\ge\eta_2$.
        \item $\mathcal{A}_N$: the set of all iterations in $\mathcal{K}_N$ with $\Xi_k\ge\overline{\tau}_\epsilon$.
        \item $\mathcal{D}_N(\Delta)$: the set of all iterations in $\mathcal{K}_N$ with $\Delta_k\ge\Delta$.
    \end{itemize}

    All complements are taken relative to $\mathcal K_N$.  Notice that $\mathcal V_N\subseteq\mathcal S_N$,
    \begin{equation}\label{ineq:metric-criticality-failure}
        \widehat{r}_k^{\mathrm{out}}\le\mu\left\|\boldsymbol{P}_k\nabla\widehat m_k(\boldsymbol{0}_p)\right\|~~~\text{for all $k\notin\mathcal C_N$},
    \end{equation}
    and
    \begin{equation}\label{eq:accepted-decrease}
        f(\boldsymbol{x}_k)-f(\boldsymbol{x}_{k+1})
        \ge\eta_1\bigl(
           \widehat m_k(\boldsymbol{0}_p)
           -\widehat m_k(\widehat{\boldsymbol{s}}_k)
           \bigr)~~~\text{for all $k\in\mathcal S_N$}.
    \end{equation}
    Moreover, as discussed in Subsection \ref{subsec:ellipsoidupdate}, RETRO has $\Delta_{k+1}=\gamma_{\mathrm{dec}}^{\mathrm{ful}}\Delta_k$ for all $k\notin\mathcal S_N$ and $\Delta_{k+1}=\Delta_k$ for all $k\in\mathcal S_N\setminus\mathcal V_N$.  For $k\in\mathcal V_N$, $\Delta_{k+1}\le\gamma_{\mathrm{inc}}^{\mathrm{ful}}\Delta_k$, with equality whenever $\gamma_{\mathrm{inc}}^{\mathrm{ful}}\Delta_k \le\mathrm{radi}_{\max}\sqrt{\underline\lambda_S}$. 

    The next lemma expresses the decrease from a successful iteration in terms of the full gradient.  The factor $\sqrt{\Xi_k}$ accounts for the portion available in the sampled ellipsoidal section.  In particular, when both $\sqrt{\Xi_k}$ and $\Delta_k$ are bounded below, a gradient norm above the target tolerance guarantees a fixed positive decrease.
    \begin{lemma}\label{lem:weighted-successful-decrease}
        Suppose that Assumptions~\ref{ass:smoothness-bdbelow} to \ref{ass:CauchyDec} hold.  If $k\in\mathcal S_N$, then
        \begin{equation*}
            f(\boldsymbol{x}_k)-f(\boldsymbol{x}_{k+1})
            \ge\frac{\eta_1\kappa_{\mathrm{cd}}}{(1+\kappa_{eg}\mu\slash\sqrt{\underline\lambda_S})}\Delta_k\sqrt{\Xi_k}\left\|\nabla f(\boldsymbol{x}_k)\right\|,
        \end{equation*}
        where
        \begin{equation*}
            \kappa_{\mathrm{cd}}:=c_{\mathrm{tr}}\min\left\{1,
            \frac{1}{\mu\sqrt{\overline\lambda_S}(\kappa_H\slash\underline\lambda_S+1)}\right\}.
        \end{equation*}
    \end{lemma}
    
    \begin{proof}
        Fix $k\in\mathcal S_N$.  Since $\widehat r_k^{\mathrm{out}}=\Delta_k\|\boldsymbol{P}_k\|$ and Inequality \eqref{ineq:shape-bounds-Sk} implies
        \begin{equation}\label{ineq:shape-bounds-Pk}
            \frac{1}{\sqrt{\overline\lambda_S}}
            \le\|\boldsymbol{P}_k\|
            \le\frac{1}{\sqrt{\underline\lambda_S}},
        \end{equation}
        Assumption~\ref{ass:bdmodelHess} gives $\|\boldsymbol{P}_k\nabla^2\widehat m_k(\boldsymbol{0}_p)\boldsymbol{P}_k\|\le\kappa_H\|\boldsymbol{P}_k\|^2 \le\kappa_H\slash\underline\lambda_S$.  Moreover, because a successful iteration passes the criticality test,
        Inequality \eqref{ineq:metric-criticality-failure} yields
        \begin{equation*}
            \left\|\boldsymbol{P}_k\nabla\widehat m_k(\boldsymbol{0}_p)\right\|\ge\frac{\widehat r_k^{\mathrm{out}}}{\mu}
            =\frac{\Delta_k\|\boldsymbol{P}_k\|}{\mu}
            \ge\frac{\Delta_k}{\mu\sqrt{\overline\lambda_S}}.
        \end{equation*}
        Therefore,
        \begin{align*}
            \min\left\{\Delta_k,\frac{\|\boldsymbol{P}_k\nabla\widehat m_k(\boldsymbol{0}_p)\|}{\|\boldsymbol{P}_k\nabla^2\widehat m_k(\boldsymbol{0}_p)\boldsymbol{P}_k\|+1}\right\}
            &\ge
            \min\left\{1,
            \frac{1}{\mu\sqrt{\overline\lambda_S}(\kappa_H\slash\underline\lambda_S+1)}\right\}\Delta_k.
        \end{align*}
        By Assumption~\ref{ass:CauchyDec} and the definition of
        $\kappa_{\mathrm{cd}}$,
        \begin{equation}\label{ineq:successful-predicted-decrease}
            \widehat m_k(\boldsymbol{0}_p)
            -\widehat m_k(\widehat{\boldsymbol{s}}_k)
            \ge\kappa_{\mathrm{cd}}\Delta_k\left\|\boldsymbol{P}_k\nabla\widehat m_k(\boldsymbol{0}_p)\right\|.
        \end{equation}

        Assumption~\ref{ass:fullylinear} gives
        \begin{align*}
            \sqrt{\Xi_k}\|\nabla f(\boldsymbol{x}_k)\|
            &=\left\|\boldsymbol{P}_k\boldsymbol{Q}_k^\top
              \nabla f(\boldsymbol{x}_k)\right\|\\
            &\le \left\|\boldsymbol{P}_k\nabla\widehat m_k(\boldsymbol{0}_p)\right\|+\|\boldsymbol{P}_k\|
              \left\|\boldsymbol{Q}_k^\top\nabla f(\boldsymbol{x}_k)
              -\nabla\widehat m_k(\boldsymbol{0}_p)\right\|\\
            &\le \left\|\boldsymbol{P}_k\nabla\widehat m_k(\boldsymbol{0}_p)\right\|+\kappa_{eg}\|\boldsymbol{P}_k\|
              \widehat r_k^{\mathrm{out}}\\
            &\le\left(1+\frac{\kappa_{eg}\mu}
              {\sqrt{\underline\lambda_S}}\right)\left\|\boldsymbol{P}_k\nabla\widehat m_k(\boldsymbol{0}_p)\right\|,
        \end{align*}
        where the last inequality uses
        $\widehat r_k^{\mathrm{out}}\le\mu\|\boldsymbol{P}_k\nabla\widehat m_k(\boldsymbol{0}_p)\|$ and $\|\boldsymbol{P}_k\|\le1/\sqrt{\underline\lambda_S}$.
        Combining this bound with Inequalities \eqref{eq:accepted-decrease} and \eqref{ineq:successful-predicted-decrease} proves the result.
    \end{proof}

    A small trust region improves model accuracy, but it can also limit the decrease available from a step.  The next lemma shows that the algorithm can increase its size if the full gradient norm is above the target tolerance and the subspace is sufficiently aligned.  In this case, a sufficiently small $\Delta_k$ ensures that the iteration passes the criticality test and is very successful.  

    \begin{lemma}\label{lem:small-Delta-verysuccess}
        Suppose that Assumptions~\ref{ass:smoothness-bdbelow} to \ref{ass:CauchyDec}
        hold, $\|\nabla f(\boldsymbol{x}_k)\|>\epsilon$, and
        $\Xi_k\ge\overline\tau_\epsilon$.  Define
        \begin{equation}\label{eq:recovery-scale}
            \Delta_*(\epsilon)
            :=\min\left\{
              \frac{\mu\underline\lambda_S}
              {2(\sqrt{\underline\lambda_S}+\mu\kappa_{eg})},
              \frac{(1-\eta_2)\kappa_{\mathrm{cd}}
              \underline\lambda_S}{4\kappa_{ef}}
              \right\}
              \sqrt{\overline\tau_\epsilon}\,\epsilon.
        \end{equation}
        If $\Delta_k\le\Delta_*(\epsilon)$, then iteration $k$ passes the criticality test
        and $\rho_k\ge\eta_2$.  In particular, $k\in\mathcal V_N$.
    \end{lemma}

    \begin{proof}
        Suppose that $\Delta_k\le\Delta_*(\epsilon)$ and denote $h_k:=\|\boldsymbol{P}_k\boldsymbol{Q}_k^\top\nabla f(\boldsymbol{x}_k)\|$.  Then, $h_k=\sqrt{\Xi_k}\|\nabla f(\boldsymbol{x}_k)\|\ge\sqrt{\overline\tau_\epsilon}\epsilon$.  The $\boldsymbol{Q}_k$-fully linear bounds and Inequality \eqref{ineq:shape-bounds-Pk} give
        \begin{equation*}
            \left\|\boldsymbol{P}_k\nabla\widehat m_k(\boldsymbol{0}_p)
            -\boldsymbol{P}_k\boldsymbol{Q}_k^\top
            \nabla f(\boldsymbol{x}_k)\right\|\le\kappa_{eg}\|\boldsymbol{P}_k\|\widehat r_k^{\mathrm{out}} = \kappa_{eg}\|\boldsymbol{P}_k\|^2\Delta_k \le \frac{\kappa_{eg}}{\underline\lambda_S}\Delta_k.
        \end{equation*}
        By definition of $\Delta_*(\epsilon)$ and $h_k\ge\sqrt{\overline\tau_\epsilon}\epsilon$, we have $\Delta_k
        \le \frac{\mu\underline\lambda_S}{2(\sqrt{\underline\lambda_S}+\mu\kappa_{eg})}h_k$ and so
        \begin{equation*}
            \left\|\boldsymbol{P}_k\nabla\widehat m_k(\boldsymbol{0}_p)-\boldsymbol{P}_k\boldsymbol{Q}_k^\top
            \nabla f(\boldsymbol{x}_k)
            \right\|\le
            \frac{\mu\kappa_{eg}}{2(\sqrt{\underline\lambda_S}+\mu\kappa_{eg})}h_k <
            \frac{h_k}{2}.
        \end{equation*}
        This gives $\|\boldsymbol{P}_k\nabla\widehat m_k(\boldsymbol{0}_p)\|>h_k\slash 2$.  By Inequality \eqref{ineq:shape-bounds-Pk},
        \begin{equation*}
            \widehat r_k^{\mathrm{out}}
            =\Delta_k\|\boldsymbol{P}_k\|\le
            \frac{\Delta_k}{\sqrt{\underline\lambda_S}} \le
            \frac{\mu\sqrt{\underline\lambda_S}}
            {2(\sqrt{\underline\lambda_S}+\mu\kappa_{eg})}h_k \le
            \frac{\mu h_k}{2} <
            \mu\left\|
            \boldsymbol{P}_k\nabla\widehat m_k(\boldsymbol{0}_p)
            \right\|.
        \end{equation*}
        Thus, iteration $k$ passes the criticality test.  The argument leading to Inequality \eqref{ineq:successful-predicted-decrease} therefore gives
        \begin{equation*}
            \widehat m_k(\boldsymbol{0}_p)
            -\widehat m_k(\widehat{\boldsymbol{s}}_k)\ge
            \kappa_{\mathrm{cd}}\Delta_k
            \left\|
            \boldsymbol{P}_k\nabla\widehat m_k(\boldsymbol{0}_p)
            \right\|\ge
            \frac{\kappa_{\mathrm{cd}}}{2}\Delta_k h_k.
        \end{equation*}
        Therefore,
        \begin{align*}
            |1-\rho_k|
            &= \frac{| f(\boldsymbol{x}_k)-f(\boldsymbol{x}_k+\boldsymbol{s}_k)-\left(\widehat m_k(\boldsymbol{0}_p)-\widehat m_k(\widehat{\boldsymbol{s}}_k)\right)|}{\widehat m_k(\boldsymbol{0}_p)-\widehat m_k(\widehat{\boldsymbol{s}}_k)}\\
            &\le\frac{| f(\boldsymbol{x}_k)-\widehat m_k(\boldsymbol{0}_p) |+|f(\boldsymbol{x}_k+\boldsymbol{s}_k)-\widehat m_k(\widehat{\boldsymbol{s}}_k)|}{\widehat m_k(\boldsymbol{0}_p)-\widehat m_k(\widehat{\boldsymbol{s}}_k)}\\
            &\le \frac{2\kappa_{ef}(\widehat r_k^{\mathrm{out}})^2}{\kappa_{\mathrm{cd}}\Delta_kh_k\slash 2} = \frac{2\kappa_{ef}\Delta_k^2\|\boldsymbol{P}_k\|^2}{\kappa_{\mathrm{cd}}\Delta_kh_k\slash 2} \le \frac{2\kappa_{ef}\Delta_k^2\slash\underline\lambda_S}{\kappa_{\mathrm{cd}}\Delta_kh_k\slash 2} = \frac{4\kappa_{ef}\Delta_k}{\kappa_{\mathrm{cd}}\underline\lambda_S h_k}.
        \end{align*}
        By definition of $\Delta_*(\epsilon)$,
        \begin{equation*}
            \Delta_k
            \le
            \frac{(1-\eta_2)\kappa_{\mathrm{cd}}
            \underline\lambda_S}{4\kappa_{ef}}
            \sqrt{\overline\tau_\epsilon}\,\epsilon
            \le
            \frac{(1-\eta_2)\kappa_{\mathrm{cd}}
            \underline\lambda_S}{4\kappa_{ef}}h_k.
        \end{equation*}
        It follows that $|1-\rho_k|\le1-\eta_2$ and therefore  $\rho_k\ge\eta_2$.  Since iteration $k$ passes the criticality test, this proves that $k\in\mathcal V_N$.
    \end{proof}

\subsection{Counting iterations}
    We now count different types of iterations through changes in $\Delta_k$.  Repeated contractions above a fixed threshold require compensating expansions, while repeated expansions below a threshold require contractions that bring the size parameter back down.  The following lemma states both relations.  
    \begin{lemma}\label{lem:scalar-level-counting}
        Let $\Delta>0$ satisfy $\gamma_{\mathrm{inc}}^{\mathrm{ful}}\Delta\le\Delta_0$ and $\gamma_{\mathrm{inc}}^{\mathrm{ful}}\Delta\le\mathrm{radi}_{\max}\sqrt{\underline\lambda_S}$.  Then, we have
        \begin{align*}
            \left|\mathcal D_N\left(\frac{\gamma_{\mathrm{inc}}^{\mathrm{ful}}\Delta}{\gamma_{\mathrm{dec}}^{\mathrm{ful}}}\right)\setminus\mathcal S_N\right|&\le C_1\left|\mathcal D_N(\Delta)\cap\mathcal V_N\right|+\frac{\log(\Delta_0/(\gamma_{\mathrm{inc}}^{\mathrm{ful}}\Delta))}{\log(1/\gamma_{\mathrm{dec}}^{\mathrm{ful}})},\\
            \left|\mathcal D_N(\Delta)^\complement\cap\mathcal V_N\right|&\le C_1^{-1}\left|\mathcal D_N\left(\frac{\gamma_{\mathrm{inc}}^{\mathrm{ful}}\Delta}{\gamma_{\mathrm{dec}}^{\mathrm{ful}}}\right)^\complement\setminus\mathcal V_N\right|,
        \end{align*}
        where
        \begin{equation*}
            C_1:=\frac{\log(\gamma_{\mathrm{inc}}^{\mathrm{ful}})}{\log(1/\gamma_{\mathrm{dec}}^{\mathrm{ful}})}.
        \end{equation*}
    \end{lemma}
    \begin{proof}
        Consider the change
        \begin{equation}\label{quant:maxlogk-maxlogk+1}
            \max\left\{0,\log\left(\frac{\Delta_{k+1}}{\gamma_{\mathrm{inc}}^{\mathrm{ful}}\Delta}\right)\right\}
            -\max\left\{0,\log\left(\frac{\Delta_k}{\gamma_{\mathrm{inc}}^{\mathrm{ful}}\Delta}\right)\right\}.
        \end{equation}
        If $k\notin\mathcal S_N$, then $\Delta_{k+1}=\gamma_{\mathrm{dec}}^{\mathrm{ful}}\Delta_k$, so Quantity~\eqref{quant:maxlogk-maxlogk+1} is nonpositive.  If additionally $\Delta_k\ge\gamma_{\mathrm{inc}}^{\mathrm{ful}}\Delta/\gamma_{\mathrm{dec}}^{\mathrm{ful}}$, then $\Delta_{k+1}\ge\gamma_{\mathrm{inc}}^{\mathrm{ful}}\Delta$, so both logarithms are nonnegative and Quantity~\eqref{quant:maxlogk-maxlogk+1} equals $-\log(1/\gamma_{\mathrm{dec}}^{\mathrm{ful}})$.  If $k\in\mathcal S_N\setminus\mathcal V_N$, then $\Delta_{k+1}=\Delta_k$, so the quantity is zero.  Thus, it is nonpositive whenever $k\notin\mathcal V_N$.
    
        Moreover, if $k\in\mathcal D_N(\Delta)^\complement\cap\mathcal V_N$, then $\Delta_{k+1}\le\gamma_{\mathrm{inc}}^{\mathrm{ful}}\Delta_k<\gamma_{\mathrm{inc}}^{\mathrm{ful}}\Delta$, so Quantity~\eqref{quant:maxlogk-maxlogk+1} is also zero.  Consequently, it can be positive only for $k\in\mathcal D_N(\Delta)\cap\mathcal V_N$.  On any such iteration, $\Delta_{k+1}\le\gamma_{\mathrm{inc}}^{\mathrm{ful}}\Delta_k$.  Notice that the function $t\mapsto\max\{0,t\}$ is nondecreasing and satisfies $\max\{0,t+c\}-\max\{0,t\}\le c$ for all $t\in\mathbb R$ and $c\ge0$.  Applying these properties with $t=\log(\Delta_k\slash(\gamma_{\mathrm{inc}}^{\mathrm{ful}}\Delta))$ and $c=\log(\gamma_{\mathrm{inc}}^{\mathrm{ful}})$ shows that Quantity~\eqref{quant:maxlogk-maxlogk+1} is at most $\log(\gamma_{\mathrm{inc}}^{\mathrm{ful}})$.
    
        Summing Quantity~\eqref{quant:maxlogk-maxlogk+1} over $k=0,\ldots,N-1$, using $\gamma_{\mathrm{inc}}^{\mathrm{ful}}\Delta\le\Delta_0$, we get
        \begin{align*}
            0&\le
            \max\left\{0,\log\left(\frac{\Delta_N}
            {\gamma_{\mathrm{inc}}^{\mathrm{ful}}\Delta}\right)\right\}\\
            &\le
            \log\left(\frac{\Delta_0}{\gamma_{\mathrm{inc}}^{\mathrm{ful}}\Delta}\right)
            -\log(1/\gamma_{\mathrm{dec}}^{\mathrm{ful}})
            \left|\mathcal D_N\left(
            \frac{\gamma_{\mathrm{inc}}^{\mathrm{ful}}\Delta}
            {\gamma_{\mathrm{dec}}^{\mathrm{ful}}}
            \right)\setminus\mathcal S_N\right|
            +\log(\gamma_{\mathrm{inc}}^{\mathrm{ful}})
            |\mathcal D_N(\Delta)\cap\mathcal V_N|.
        \end{align*}
        Rearranging proves the first inequality.
    
        For the second inequality, we apply the same summation argument to
        \begin{equation}\label{quant:reverse-maxlogk-maxlogk+1}
            \max\left\{0,\log\left(\frac{\gamma_{\mathrm{inc}}^{\mathrm{ful}}\Delta}{\Delta_{k+1}}\right)\right\}
            -\max\left\{0,\log\left(\frac{\gamma_{\mathrm{inc}}^{\mathrm{ful}}\Delta}{\Delta_k}\right)\right\}.
        \end{equation}
        By Inequality~\eqref{ineq:shape-bounds-Sk}, we have $\mathrm{radi}_{\max}\sqrt{\lambda_{\min}(\boldsymbol S_{k+1})}\ge\mathrm{radi}_{\max}\sqrt{\underline\lambda_S}\ge\gamma_{\mathrm{inc}}^{\mathrm{ful}}\Delta$.  Hence, if $k\in\mathcal D_N(\Delta)^\complement\cap\mathcal V_N$, Equation~\eqref{eq:outer-cap} gives $\Delta_{k+1}=\gamma_{\mathrm{inc}}^{\mathrm{ful}}\Delta_k<\gamma_{\mathrm{inc}}^{\mathrm{ful}}\Delta$.  Both logarithms are then positive, so Quantity~\eqref{quant:reverse-maxlogk-maxlogk+1} equals $-\log(\gamma_{\mathrm{inc}}^{\mathrm{ful}})$.  If instead $k\in\mathcal D_N(\Delta)\cap\mathcal V_N$, both entries in the minimum in Equation~\eqref{eq:outer-cap} are at least $\gamma_{\mathrm{inc}}^{\mathrm{ful}}\Delta$, so the quantity is nonpositive.
    
        For every $k\notin\mathcal V_N$, the contraction and keep updates give $\gamma_{\mathrm{dec}}^{\mathrm{ful}}\Delta_k\le\Delta_{k+1}\le\Delta_k$.  The same properties of $t\mapsto\max\{0,t\}$ therefore bound Quantity~\eqref{quant:reverse-maxlogk-maxlogk+1} by $\log(1/\gamma_{\mathrm{dec}}^{\mathrm{ful}})$.  Moreover, if $\Delta_k\ge\gamma_{\mathrm{inc}}^{\mathrm{ful}}\Delta/\gamma_{\mathrm{dec}}^{\mathrm{ful}}$, then $\Delta_{k+1}\ge\gamma_{\mathrm{inc}}^{\mathrm{ful}}\Delta$, so both terms are zero.  Thus, the quantity can be positive only for $k\in\mathcal D_N(\gamma_{\mathrm{inc}}^{\mathrm{ful}}\Delta/\gamma_{\mathrm{dec}}^{\mathrm{ful}})^\complement\setminus\mathcal V_N$.
    
        Since $\gamma_{\mathrm{inc}}^{\mathrm{ful}}\Delta\le\Delta_0$, the initial maximum in this summation is zero.  Summing Quantity~\eqref{quant:reverse-maxlogk-maxlogk+1} therefore gives
        \begin{align*}
            0&\le
            \max\left\{0,\log\left(
            \frac{\gamma_{\mathrm{inc}}^{\mathrm{ful}}\Delta}{\Delta_N}
            \right)\right\}\\
            &\le
            -\log(\gamma_{\mathrm{inc}}^{\mathrm{ful}})
            |\mathcal D_N(\Delta)^\complement\cap\mathcal V_N|
            +\log(1/\gamma_{\mathrm{dec}}^{\mathrm{ful}})
            \left|\mathcal D_N\left(
            \frac{\gamma_{\mathrm{inc}}^{\mathrm{ful}}\Delta}
            {\gamma_{\mathrm{dec}}^{\mathrm{ful}}}
            \right)^\complement\setminus\mathcal V_N\right|.
        \end{align*}
        Rearranging proves the second inequality.
    \end{proof}

    To combine this count with Lemma~\ref{lem:small-Delta-verysuccess}, we choose a threshold that satisfies the conditions of Lemma~\ref{lem:scalar-level-counting}, while also lying below the threshold from Lemma~\ref{lem:small-Delta-verysuccess}.  In particular, for $\epsilon>0$, consider
    \begin{equation}\label{eq:zero-analysis-scale}
        \Delta_\epsilon:=\min\left\{
        \frac{\Delta_0}{\gamma_{\mathrm{inc}}^{\mathrm{ful}}},
        \frac{\gamma_{\mathrm{dec}}^{\mathrm{ful}}}{\gamma_{\mathrm{inc}}^{\mathrm{ful}}}\Delta_*(\epsilon),
        \frac{\mathrm{radi}_{\max}\sqrt{\underline\lambda_S}}{\gamma_{\mathrm{inc}}^{\mathrm{ful}}}
        \right\}.
    \end{equation}
    Thus $\gamma_{\mathrm{inc}}^{\mathrm{ful}}\Delta_\epsilon\le\Delta_0$ and $\gamma_{\mathrm{inc}}^{\mathrm{ful}}\Delta_\epsilon\le\mathrm{radi}_{\max}\sqrt{\underline\lambda_S}$, so Lemma~\ref{lem:scalar-level-counting} applies with $\Delta=\Delta_\epsilon$.  Moreover, we have $\gamma_{\mathrm{inc}}^{\mathrm{ful}}\Delta_\epsilon/\gamma_{\mathrm{dec}}^{\mathrm{ful}}\le\Delta_*(\epsilon)$.

    The next lemma bounds the number of sufficiently aligned iterations that can occur while all gradient norms remain above $\epsilon$.  Successful aligned iterations with $\Delta_k\ge\Delta_\epsilon$ are limited by the available decrease in $f$.  Aligned iterations with smaller $\Delta_k$ are very successful, so the update counts control how often they can occur.  Combining these observations gives a bound consisting of a term depending on $\epsilon$ and a fixed fraction of the total number of iterations.
    \begingroup
    \setlength{\emergencystretch}{1em}
    \begin{lemma}\label{lem:zero-aligned-count}
        Suppose that Assumptions~\ref{ass:smoothness-bdbelow} to \ref{ass:CauchyDec} hold and that $\|\nabla f(\boldsymbol{x}_k)\|>\epsilon$ for every $k\in\mathcal K_N$.  Then,
        \begin{equation*}
            \left|\mathcal A_N\right|\le \phi_\epsilon+\frac{1}{C_\gamma+1}\frac{\log(\Delta_0/(\gamma_{\mathrm{inc}}^{\mathrm{ful}}\Delta_\epsilon))}{\log(1/\gamma_{\mathrm{dec}}^{\mathrm{ful}})}+\frac{C_\gamma}{C_\gamma+1}N,
        \end{equation*}
        where $C_\gamma:=\max\{C_1,C_1^{-1}\}$ and
        \begin{equation*}
            \phi_\epsilon:=\frac{(f(\boldsymbol{x}_0)-f_{\mathrm{low}})(1+\kappa_{eg}\mu\slash\sqrt{\underline\lambda_S})}{\eta_1\kappa_{\mathrm{cd}}\Delta_\epsilon\sqrt{\overline\tau_\epsilon}\,\epsilon}.
        \end{equation*}
    \end{lemma}
    \endgroup

    \begin{proof}
        Since $\gamma_{\mathrm{inc}}^{\mathrm{ful}}\Delta_\epsilon/\gamma_{\mathrm{dec}}^{\mathrm{ful}}\le\Delta_*(\epsilon)$, Lemma~\ref{lem:small-Delta-verysuccess} shows that every iteration in $\mathcal{A}_N$ with $\Delta_k<\gamma_{\mathrm{inc}}^{\mathrm{ful}}\Delta_\epsilon/\gamma_{\mathrm{dec}}^{\mathrm{ful}}$ is very successful.  Also, every iteration in $\mathcal A_N\cap\mathcal S_N\cap\mathcal D_N(\Delta_\epsilon)$ gives at least the decrease in Lemma~\ref{lem:weighted-successful-decrease} with $\Delta_k\ge\Delta_\epsilon$.  Since every other iteration either decreases $f$ or leaves it unchanged, we have that
        \begin{align*}
            f(\boldsymbol{x}_0)-f_{\mathrm{low}}
            \ge f(\boldsymbol{x}_0)-f(\boldsymbol{x}_N) &\ge\sum_{k\in\mathcal A_N\cap\mathcal S_N\cap\mathcal D_N(\Delta_\epsilon)}
            \bigl(f(\boldsymbol{x}_k)-f(\boldsymbol{x}_{k+1})\bigr)\\
            &\ge
            \frac{\eta_1\kappa_{\mathrm{cd}}\Delta_\epsilon
            \sqrt{\overline\tau_\epsilon}\,\epsilon}
            {1+\kappa_{eg}\mu\slash\sqrt{\underline\lambda_S}}
            \left|\mathcal A_N\cap\mathcal S_N\cap\mathcal D_N(\Delta_\epsilon)\right|,
        \end{align*}
        and hence $|\mathcal A_N\cap\mathcal S_N\cap\mathcal D_N(\Delta_\epsilon)|\le\phi_\epsilon$.

        Now, we consider how the iterations in $\mathcal A_N$ are classified.  Let $k\in\mathcal A_N$. If $k\in\mathcal S_N$ and $\Delta_k\ge\Delta_\epsilon$, then $k\in\mathcal A_N\cap\mathcal S_N\cap\mathcal D_N(\Delta_\epsilon)$. If instead $k\in\mathcal S_N$ and $\Delta_k<\Delta_\epsilon$, then we have $\Delta_k<\Delta_\epsilon<\gamma_{\mathrm{inc}}^{\mathrm{ful}}\Delta_\epsilon/\gamma_{\mathrm{dec}}^{\mathrm{ful}}$. Lemma~\ref{lem:small-Delta-verysuccess} therefore implies that $k\in\mathcal V_N$, and hence $k\in\mathcal D_N(\Delta_\epsilon)^\complement\cap\mathcal V_N$. It remains to consider $k\in\mathcal A_N\setminus\mathcal S_N$. Such an iteration cannot satisfy $\Delta_k<\gamma_{\mathrm{inc}}^{\mathrm{ful}}\Delta_\epsilon/\gamma_{\mathrm{dec}}^{\mathrm{ful}}$, because Lemma~\ref{lem:small-Delta-verysuccess} would then imply $k\in\mathcal V_N\subseteq\mathcal S_N$, contradicting $k\notin\mathcal S_N$. Consequently, every $k\in\mathcal A_N\setminus\mathcal S_N$ satisfies $\Delta_k\ge\gamma_{\mathrm{inc}}^{\mathrm{ful}}\Delta_\epsilon/\gamma_{\mathrm{dec}}^{\mathrm{ful}}$, and therefore belongs to $\mathcal D_N(\gamma_{\mathrm{inc}}^{\mathrm{ful}}\Delta_\epsilon/\gamma_{\mathrm{dec}}^{\mathrm{ful}})\setminus\mathcal S_N$. Thus every iteration in $\mathcal A_N$ is contained in one of $\mathcal A_N\cap\mathcal S_N\cap\mathcal D_N(\Delta_\epsilon)$, $\mathcal D_N(\gamma_{\mathrm{inc}}^{\mathrm{ful}}\Delta_\epsilon/\gamma_{\mathrm{dec}}^{\mathrm{ful}})\setminus\mathcal S_N$, or $\mathcal D_N(\Delta_\epsilon)^\complement\cap\mathcal V_N$. Applying Lemma~\ref{lem:scalar-level-counting} and using $|\mathcal A_N\cap\mathcal S_N\cap\mathcal D_N(\Delta_\epsilon)|\le\phi_\epsilon$ gives 
        \begin{equation*}
            \left|\mathcal A_N\right|\le\phi_\epsilon+C_1\left|\mathcal D_N(\Delta_\epsilon)\cap\mathcal V_N\right|+C_1^{-1}\left|\mathcal D_N\left(\frac{\gamma_{\mathrm{inc}}^{\mathrm{ful}}\Delta_\epsilon}{\gamma_{\mathrm{dec}}^{\mathrm{ful}}}\right)^\complement\setminus\mathcal V_N\right|+\frac{\log(\Delta_0/(\gamma_{\mathrm{inc}}^{\mathrm{ful}}\Delta_\epsilon))}{\log(1/\gamma_{\mathrm{dec}}^{\mathrm{ful}})}.
        \end{equation*}
        Notice that the two sets multiplied by $C_1$ and $C_1^{-1}$ are disjoint.  Moreover, any iteration in $\mathcal A_N\cap\mathcal D_N(\Delta_\epsilon)\cap\mathcal V_N$ is very successful and hence successful, so it belongs to $\mathcal A_N\cap\mathcal S_N\cap\mathcal D_N(\Delta_\epsilon)$. Therefore at most $\phi_\epsilon$ iterations in $\mathcal A_N$ belong to $\mathcal D_N(\Delta_\epsilon)\cap\mathcal V_N$. On the other hand, $\mathcal D_N(\gamma_{\mathrm{inc}}^{\mathrm{ful}}\Delta_\epsilon\slash\gamma_{\mathrm{dec}}^{\mathrm{ful}})^\complement\setminus\mathcal V_N$ contains no iteration of $\mathcal A_N$.  Indeed, if $k\in\mathcal A_N$ and $\Delta_k<\gamma_{\mathrm{inc}}^{\mathrm{ful}}\Delta_\epsilon/\gamma_{\mathrm{dec}}^{\mathrm{ful}}$, then Lemma~\ref{lem:small-Delta-verysuccess} implies $k\in\mathcal V_N$. It follows that, among the iterations in the union of these two disjoint sets, at most $\phi_\epsilon$ belong to $\mathcal A_N$, while all remaining iterations lie outside $\mathcal A_N$. Hence, the cardinality of their union is at most $\phi_\epsilon+N-|\mathcal A_N|$. Since $C_\gamma=\max\{C_1,C_1^{-1}\}$, we obtain 
        \begin{equation*}
            |\mathcal A_N|\le\phi_\epsilon+C_\gamma(\phi_\epsilon+N-|\mathcal A_N|)+\frac{\log(\Delta_0/(\gamma_{\mathrm{inc}}^{\mathrm{ful}}\Delta_\epsilon))}{\log(1/\gamma_{\mathrm{dec}}^{\mathrm{ful}})}.
        \end{equation*}
        Rearranging and dividing by $C_\gamma+1$ yields the claimed result.
    \end{proof}

\subsection{Complexity and almost-sure convergence}
    Lemma~\ref{lem:zero-aligned-count} gives an upper bound on the number of sufficiently aligned iterations before a small gradient is reached.  The sampling assumption makes a much smaller than expected number of such iterations unlikely.  When the conditional alignment probability exceeds the fraction in the counting bound, these two facts limit how long the algorithm can continue without reaching the target tolerance.  The following theorem gives an exponentially decreasing bound on that probability and a bound on the expected iteration count.  If the alignment threshold $\overline\tau_\epsilon$ and the gap $\pi_\epsilon-C_\gamma/(C_\gamma+1)$ remain bounded below by positive constants as $\epsilon$ decreases, the expected iteration count is $\mathcal O(\epsilon^{-2})$.  In Subsection~\ref{subsec:subspacesampling}, we provide two sampling strategies that satisfy these requirements.
    \begin{theorem}\label{thm:zero-radius-complexity}
        Suppose that RETRO is run with $\mathrm{radi}_{\min}=0$.  Suppose that Assumptions~\ref{ass:smoothness-bdbelow} to \ref{ass:mean-xi} hold with
        \begin{equation}\label{eq:zero-probability-condition}
            \pi_\epsilon>\frac{C_\gamma}{C_\gamma+1}.
        \end{equation}
        Let $K_\epsilon:=\inf\{k\ge0:\|\nabla f(\boldsymbol{x}_k)\|\le\epsilon\}$.  For every positive integer $N$ satisfying
        \begin{equation*}
            N\ge\frac{2}{\pi_\epsilon-C_\gamma/(C_\gamma+1)}\left(\phi_\epsilon+\frac{1}{C_\gamma+1}\frac{\log(\Delta_0/(\gamma_{\mathrm{inc}}^{\mathrm{ful}}\Delta_\epsilon))}{\log(1/\gamma_{\mathrm{dec}}^{\mathrm{ful}})}\right),
        \end{equation*}
        one has 
        \begin{equation}\label{ineq:Kepstailbound}
            \mathbb P(K_\epsilon\ge N)\le\exp\left(-\frac{(\pi_\epsilon-C_\gamma/(C_\gamma+1))^2}{8\pi_\epsilon}N\right).
        \end{equation}
        Taking $N$ to be the smallest positive integer satisfying this lower bound, we also have
        \begin{equation}\label{ineq:Kepsexpectationbound}
            \mathbb E[K_\epsilon]\le N+\frac{\exp\left(-\frac{(\pi_\epsilon-C_\gamma/(C_\gamma+1))^2}{8\pi_\epsilon}N\right)}{1-\exp\left(-\frac{(\pi_\epsilon-C_\gamma/(C_\gamma+1))^2}{8\pi_\epsilon}\right)}.
        \end{equation}
        If $\overline\tau_\epsilon$ and $\pi_\epsilon-C_\gamma/(C_\gamma+1)$ are bounded below by positive constants independent of $\epsilon$, then $\mathbb E[K_\epsilon]=\mathcal O(\epsilon^{-2})$.
    \end{theorem}

    \begin{proof}
        For each $k$, let $I_k$ be the indicator of the event $\{\|\nabla f(\boldsymbol{x}_k)\|\le\epsilon\}\cup\{\Xi_k\ge\overline\tau_\epsilon\}$.  On $\{\|\nabla f(\boldsymbol{x}_k)\|\le\epsilon\}$, the indicator equals one, while on its complement, Assumption~\ref{ass:mean-xi} gives conditional probability at least $\pi_\epsilon$ that $\Xi_k\ge\overline\tau_\epsilon$.  Therefore, we have that $\mathbb P(I_k=1\mid\mathcal F_{k-1})\ge\pi_\epsilon$.  Applying \cite[Lemma~4.5]{gratton2015direct} with $p=\pi_\epsilon$, $\lambda=\frac12(\pi_\epsilon+\frac{C_\gamma}{C_\gamma+1})\in(0,\pi_\epsilon)$, and $k=N$, we obtain
        \begin{equation*}
            \mathbb{P}\left(\sum_{k=0}^{N-1} I_k\le\frac12\left(\pi_\epsilon+\frac{C_\gamma}{C_\gamma+1}\right)N\right)\le\exp\left(-\frac{\left(\pi_\epsilon-C_\gamma/(C_\gamma+1)\right)^2}{8\pi_\epsilon}N\right).
        \end{equation*}

        On $\{K_\epsilon\ge N\}$, all gradient norms at iterations $0,\ldots,N-1$ exceed $\epsilon$, so the sum of these indicators equals $|\mathcal A_N|$.  For $N$ satisfying the stated lower bound, Lemma~\ref{lem:zero-aligned-count} gives $|\mathcal A_N|\le\frac12(\pi_\epsilon+C_\gamma/(C_\gamma+1))N$.  Thus, we have
        \begin{equation*}
            \left\{K_\epsilon\ge N\right\}\subseteq\left\{\sum_{k=0}^{N-1} I_k\le\frac12\left(\pi_\epsilon+\frac{C_\gamma}{C_\gamma+1}\right)N\right\},
        \end{equation*}
        proving Inequality \eqref{ineq:Kepstailbound}.  Inequality \eqref{ineq:Kepsexpectationbound} follows from
        \begin{align*}
            \mathbb E[K_\epsilon] &=\sum_{j=1}^{\infty}\mathbb P(K_\epsilon\ge j)\\
            &\le N+\sum_{j=N}^{\infty}\mathbb P(K_\epsilon\ge j)\le N+\sum_{j=N}^{\infty}\exp\left(-\frac{(\pi_\epsilon-C_\gamma/(C_\gamma+1))^2}{8\pi_\epsilon}j\right)\\
            &\le N+\frac{\exp\left(-\frac{(\pi_\epsilon-C_\gamma/(C_\gamma+1))^2}{8\pi_\epsilon}N\right)}{1-\exp\left(-\frac{(\pi_\epsilon-C_\gamma/(C_\gamma+1))^2}{8\pi_\epsilon}\right)}.
        \end{align*}

        Finally, Equations \eqref{eq:recovery-scale} and \eqref{eq:zero-analysis-scale} imply that, for sufficiently small $\epsilon$, $\Delta_\epsilon$ is bounded below by a positive constant times $\sqrt{\overline\tau_\epsilon}\,\epsilon$.  Therefore $\phi_\epsilon=\mathcal O((\overline\tau_\epsilon\epsilon^2)^{-1})$, while the logarithmic term in the lower bound on $N$ is $\mathcal O(\log(1/\epsilon))$ when $\overline\tau_\epsilon$ is uniformly bounded below.  If $\pi_\epsilon-C_\gamma/(C_\gamma+1)$ is also uniformly bounded below, the required $N$ is $\mathcal O(\epsilon^{-2})$ and the second term on the right-hand side of Inequality \eqref{ineq:Kepsexpectationbound} is bounded above by a constant independent of $\epsilon$.  Hence, $\mathbb E[K_\epsilon]=\mathcal O(\epsilon^{-2})$.
    \end{proof}

    The probability bound can also be applied starting from any later iteration.  If its assumptions hold for a sequence of tolerances tending to zero, the algorithm reaches arbitrarily small gradient norms at arbitrarily late iterations with probability one.  This gives the following first-order liminf convergence result.
    \begin{theorem}\label{cor:as-liminf}
        Suppose that RETRO is run with $\mathrm{radi}_{\min}=0$.  Suppose that the assumptions of Theorem~\ref{thm:zero-radius-complexity} hold for every $\epsilon=1/m$, $m=1,2,\ldots$.  Then,
        \begin{equation*}
            \mathbb P\left(\liminf_{k\to\infty}\|\nabla f(\boldsymbol{x}_k)\|=0\right)=1.
        \end{equation*}
    \end{theorem}

    \begin{proof}
        Fix integers $\ell\ge0$ and $m\ge1$, and define
        \begin{equation*}
            K_{\ell,m}:=\inf\left\{k\ge\ell:\|\nabla f(\boldsymbol{x}_k)\|\le\frac{1}{m}\right\}.
        \end{equation*}
        Conditional on $\mathcal F_{\ell-1}$, we apply the argument in the proof of Theorem~\ref{thm:zero-radius-complexity} starting from iteration $\ell$, with $\epsilon=1/m$.  Here, $\boldsymbol{x}_0$ and $\Delta_0$ in the counting bound are replaced by $\boldsymbol{x}_\ell$ and $\Delta_\ell$, so the resulting lower bound on $N$ is finite and $\mathcal F_{\ell-1}$-measurable.  For all positive integers $N$ satisfying this bound, we have
        \begin{equation*}
            \mathbb P\left(K_{\ell,m}\ge\ell+N\mid\mathcal F_{\ell-1}\right)
            \le \exp\left(-\frac{(\pi_{1/m}-C_\gamma/(C_\gamma+1))^2}{8\pi_{1/m}}N\right).
        \end{equation*}
        Since $\pi_{1/m}>C_\gamma/(C_\gamma+1)$, letting $N\to\infty$, by the continuity of conditional probability, we get $\mathbb P(K_{\ell,m}<\infty\mid\mathcal F_{\ell-1})=1$.  By the law of total expectation, $\mathbb P(K_{\ell,m}<\infty)=\mathbb E[\mathbb P(K_{\ell,m}<\infty\mid\mathcal F_{\ell-1})]=1$ for every $\ell\ge0$ and $m\ge1$.

        Taking the countable intersection over all $\ell\ge0$ and $m\ge1$, with probability one, for every $\ell$ and $m$ there exists some $k\ge\ell$ such that $\|\nabla f(\boldsymbol{x}_k)\|\le 1\slash m$. It follows that $\liminf_{k\to\infty}\|\nabla f(\boldsymbol{x}_k)\|\le 1\slash m$ for every $m\ge1$ and the result follows.
    \end{proof}

    \begin{remark}
        Under the probabilistic $\alpha$-well-alignment condition used in \cite{cartis2023scalable,chen2024qfully,dzahini2024stochastic} (see Inequality~\eqref{eq:ordinary-alignment}), Inequality~\eqref{eq:zero-probability-condition} becomes
        $\delta_S<1/(C_\gamma+1)$. Thus,
        Theorems~\ref{thm:zero-radius-complexity} and \ref{cor:as-liminf}
        apply under this condition and their remaining assumptions.
    \end{remark}

    \begin{remark}\label{rem:radius-factor-choice}
        Since $C_\gamma=\max\{C_1,C_1^{-1}\}\ge1$, the right-hand side of Inequality~\eqref{eq:zero-probability-condition} is minimized when $C_1=1$, equivalently when $\gamma_{\mathrm{inc}}^{\mathrm{ful}}=1/\gamma_{\mathrm{dec}}^{\mathrm{ful}}$.  In that case $C_\gamma=1$, and the condition reduces to $\pi_\epsilon>1/2$.
    \end{remark}

\subsection{Subspace sampling}\label{subsec:subspacesampling}
    The convergence and complexity results above are stated under Assumption~\ref{ass:mean-xi} and Inequality~\eqref{eq:zero-probability-condition}, which requires the sampled subspace to achieve a prescribed level of ellipsoidal alignment with sufficiently high conditional probability.  We now show that this can be satisfied by concrete randomized subspace constructions. We consider two sampling strategies: Haar sampling, which draws the subspace uniformly without favoring any direction in the ambient Euclidean geometry, and metric-guided sampling, which incorporates the current ellipsoidal shape matrix into the sampling distribution. For each strategy, we derive an explicit alignment threshold and a corresponding probability bound, and show how these quantities can be chosen to satisfy the probability condition required in Theorem~\ref{thm:zero-radius-complexity}.

\subsubsection{Haar sampling}
    We first consider Haar sampling, which provides a natural isotropic baseline for subspace selection. At each iteration, $\boldsymbol{Q}_k$ is drawn uniformly from the set of $p$-dimensional orthonormal frames, independently of any preferred direction induced by the current ellipsoid. Equivalently, $\boldsymbol{Q}_k$ may be obtained as the orthonormal factor in a thin QR factorization of an $n\times p$ matrix with independent standard normal entries, with the diagonal entries of the triangular factor chosen positive \cite{chikuse2003statistics}. To verify Assumption~\ref{ass:mean-xi} and Inequality~\eqref{eq:zero-probability-condition} for this sampling rule, we first recall the distribution of the squared projection of a fixed vector onto a Haar-uniform subspace.
    
    \begin{lemma}\label{lem:Haar-beta}
        Let $\boldsymbol{U}\in\mathbb R^{n\times p}$ be Haar-uniform with orthonormal columns, where $1\le p\le n$, and let $\boldsymbol{v}$ be a fixed unit vector.  If $p<n$, then $\|\boldsymbol{U}^\top\boldsymbol{v}\|^2\sim\mathrm{Beta}(p\slash 2,(n-p)\slash 2)$; for $p=n$, the quantity equals one.  In all cases, we have
        \begin{equation}\label{eq:Haar-moments}
            \mathbb E\|\boldsymbol{U}^\top\boldsymbol{v}\|^2=\frac pn~~~\text{and}~~~\mathbb E\|\boldsymbol{U}^\top\boldsymbol{v}\|^4=\frac{p(p+2)}{n(n+2)}.
        \end{equation}
    \end{lemma}
    
    \begin{proof}
        By rotational invariance of Haar-uniform matrices~\cite{chikuse2003statistics}, the distribution of $\|\boldsymbol{U}^\top\boldsymbol{v}\|^2$ is the same as the squared projection of a uniformly distributed unit vector onto $\mathrm{span}(\boldsymbol{e}^1,\ldots,\boldsymbol{e}^p)$.  Write this unit vector as $\boldsymbol{Z}/\|\boldsymbol{Z}\|$, where the components of $\boldsymbol{Z}$ are independent standard normal variables.  The squared projection is therefore $(Z_1^2+\cdots+Z_p^2)\slash(Z_1^2+\cdots+Z_n^2)$.  The numerator and the remaining sum $Z_{p+1}^2+\cdots+Z_n^2$ are independent chi-square variables, which gives the beta law and the moments in \eqref{eq:Haar-moments}.
    \end{proof}
    
    We first verify that Haar sampling satisfies Assumption~\ref{ass:mean-xi}.
    Suppose now that, conditionally on $\mathcal F_{k-1}$,
    $\boldsymbol{Q}_k$ is Haar-uniform. Whenever
    $\|\nabla f(\boldsymbol{x}_k)\|>\epsilon$, the normalized gradient is fixed
    after conditioning, since $\boldsymbol{x}_k$ is
    $\mathcal F_{k-1}$-measurable. Lemma~\ref{lem:Haar-beta} therefore shows that
    $\|\boldsymbol{Q}_k^\top\nabla f(\boldsymbol{x}_k)\|^2/
    \|\nabla f(\boldsymbol{x}_k)\|^2$ has the conditional distribution of
    $Z\sim\mathrm{Beta}(p/2,(n-p)/2)$ for $p<n$, while $Z=1$ for $p=n$.
    Hence, by Inequality~\eqref{eq:Xi-lower}, for every
    $\vartheta\in(0,1)$, Assumption~\ref{ass:mean-xi} holds with
    \begin{equation}\label{eq:Haar-threshold-constants}
        \overline\tau_\epsilon
        =\frac{\vartheta p}{n\overline\lambda_S}~~~\text{and}~~~\pi_\epsilon
        =\mathbb P\left(Z\ge\frac{\vartheta p}{n}\right)
        \ge(1-\vartheta)^2\frac{p(n+2)}{n(p+2)},
    \end{equation}
    because
    \begin{equation*}
        \mathbb P(\Xi_k\ge\overline\tau_\epsilon\mid\mathcal F_{k-1})
        \ge\mathbb P\left(
            \frac{\|\boldsymbol{Q}_k^\top\nabla f(\boldsymbol{x}_k)\|^2}
                 {\|\nabla f(\boldsymbol{x}_k)\|^2}
            \ge\frac{\vartheta p}{n}
            \ \middle|\ \mathcal F_{k-1}\right)
        =\pi_\epsilon.
    \end{equation*}
    For the explicit lower bound on $\pi_\epsilon$, note that Cauchy-Schwarz gives
    $$(1-\vartheta)\mathbb E[Z]\le\mathbb E[Z\mathbf1_{\{Z\ge\vartheta\mathbb E[Z]\}}]
    \le\sqrt{\mathbb E[Z^2]\mathbb P(Z\ge\vartheta\mathbb E[Z])}.$$
    Substituting the moments in Lemma~\ref{lem:Haar-beta} yields the bound.

    The lower bound $(1-\vartheta)^2p(n+2)/(n(p+2))$ need not be large enough to satisfy $\pi_\epsilon>C_\gamma/(C_\gamma+1)$. However, the probability $\pi_\epsilon=\mathbb P(Z\ge\vartheta p/n)$ tends to one as $\vartheta\downarrow0$, since $Z>0$ almost surely. Thus, for fixed $n$ and $p$, a sufficiently small fixed $\vartheta\in(0,1)$ ensures the required inequality, with both $\overline\tau_\epsilon$ and $\pi_\epsilon$ independent of $k$ and $\epsilon$.

\subsubsection{Metric-guided sampling}
    Haar sampling treats all Euclidean directions equally and therefore does not exploit the anisotropy encoded by the current ellipsoid. A natural alternative is to bias the sampled subspace according to the shape matrix $\boldsymbol{S}_k$. Since directions corresponding to smaller eigenvalues of $\boldsymbol{S}_k$ are less restricted by the ellipsoidal trust region, metric-guided sampling gives these directions greater representation in the sampled subspace. Specifically, conditionally on $\mathcal F_{k-1}$, we draw $\boldsymbol{U}_k\in\mathbb R^{n\times p}$ Haar-uniform, transform its columns by $\boldsymbol{S}_k^{-1/2}$, and then orthonormalize the resulting subspace. Thus, we set
    \begin{equation}\label{eq:metric-guided-sampling}
        \boldsymbol{A}_k:=\boldsymbol{S}_k^{-1/2}\boldsymbol{U}_k,
    \end{equation}
    and choose $\boldsymbol{Q}_k$ as an orthonormal basis for $\mathrm{col}(\boldsymbol{A}_k)$. To interpret this construction, write a step as $\boldsymbol{s}=\boldsymbol{S}_k^{-1/2}\boldsymbol{y}$. The ellipsoidal constraint then becomes $\|\boldsymbol{y}\|\le\Delta_k$, so we sample a subspace uniformly in the $\boldsymbol{y}$ coordinates, with no preferred direction. Multiplying its basis $\boldsymbol{U}_k$ by $\boldsymbol{S}_k^{-1/2}$ maps the sampled directions back to the original coordinates and gives $\boldsymbol{A}_k$. We then orthonormalize these directions to obtain $\boldsymbol{Q}_k$, without changing the subspace they span.

    The next result characterizes the ellipsoidal subspace alignment for this sampling rule. It shows that the random projection component has the same beta distribution as in the Haar case, while the effect of the current ellipsoid enters through a deterministic factor depending on the orientation of the gradient relative to the shape matrix. This yields both a geometry-dependent description of the typical alignment and, using the uniform eigenvalue bounds on $\boldsymbol{S}_k$, an iteration-independent lower bound sufficient to verify Assumption~\ref{ass:mean-xi} and Inequality~\eqref{eq:zero-probability-condition}.
    \begin{proposition}\label{prop:metric-guided-beta}
        Under metric-guided sampling, whenever
        $\nabla f(\boldsymbol{x}_k)\ne\boldsymbol{0}_n$,
        \begin{equation}\label{eq:guided-alignment-distribution}
            \Xi_k\mid\mathcal F_{k-1}\ \overset{d}{=}
            \frac{\nabla f(\boldsymbol{x}_k)^\top\boldsymbol{S}_k^{-1}\nabla f(\boldsymbol{x}_k)}
                 {\|\nabla f(\boldsymbol{x}_k)\|^2}\,Z~~~\text{and}~~~\mathbb E[\Xi_k\mid\mathcal F_{k-1}]
            =\frac pn\,
            \frac{\nabla f(\boldsymbol{x}_k)^\top\boldsymbol{S}_k^{-1}\nabla f(\boldsymbol{x}_k)}
                 {\|\nabla f(\boldsymbol{x}_k)\|^2},
        \end{equation}
        where $Z\sim\mathrm{Beta}(p/2,(n-p)/2)$ for $p<n$ and $Z=1$ for
        $p=n$. Under Inequality~\eqref{ineq:shape-bounds-Sk},
        Assumption~\ref{ass:mean-xi} therefore holds for every $\epsilon>0$
        with the constants in Equations~\eqref{eq:Haar-threshold-constants} for any fixed
        $\vartheta\in(0,1)$.
    \end{proposition}
    \begin{proof}
        Conditionally on $\mathcal F_{k-1}$, $\boldsymbol{S}_k$ and $\nabla f(\boldsymbol{x}_k)$ are fixed, while $\boldsymbol{U}_k$ is Haar-uniform.  Since $\boldsymbol{A}_k$ has full column rank and $\boldsymbol{Q}_k$ is an orthonormal basis for its column space, $\boldsymbol{R}_k:=\boldsymbol{Q}_k^\top\boldsymbol{A}_k\in\mathbb R^{p\times p}$ is invertible and $\boldsymbol{A}_k=\boldsymbol{Q}_k\boldsymbol{R}_k$.  Moreover, $\boldsymbol{A}_k^\top\boldsymbol{S}_k\boldsymbol{A}_k=\boldsymbol{U}_k^\top\boldsymbol{U}_k=\boldsymbol{I}_p$.  Therefore,
        \begin{align*}
            \boldsymbol{Q}_k(\boldsymbol{Q}_k^\top\boldsymbol{S}_k\boldsymbol{Q}_k)^{-1}\boldsymbol{Q}_k^\top &=\boldsymbol{Q}_k\boldsymbol{R}_k
            (\boldsymbol{R}_k^\top\boldsymbol{Q}_k^\top\boldsymbol{S}_k\boldsymbol{Q}_k\boldsymbol{R}_k)^{-1}
            \boldsymbol{R}_k^\top\boldsymbol{Q}_k^\top\\
            &=\boldsymbol{A}_k(\boldsymbol{A}_k^\top\boldsymbol{S}_k\boldsymbol{A}_k)^{-1}\boldsymbol{A}_k^\top
            =\boldsymbol{A}_k\boldsymbol{A}_k^\top
            =\boldsymbol{S}_k^{-1/2}\boldsymbol{U}_k\boldsymbol{U}_k^\top\boldsymbol{S}_k^{-1/2}.
        \end{align*}
        Hence,
        \begin{equation*}
            \Xi_k
            =\frac{\|\boldsymbol{U}_k^\top\boldsymbol{S}_k^{-1/2}
                         \nabla f(\boldsymbol{x}_k)\|^2}
                   {\|\boldsymbol{S}_k^{-1/2}\nabla f(\boldsymbol{x}_k)\|^2}
             \frac{\|\boldsymbol{S}_k^{-1/2}\nabla f(\boldsymbol{x}_k)\|^2}
                  {\|\nabla f(\boldsymbol{x}_k)\|^2}.
        \end{equation*}
        Applying Lemma~\ref{lem:Haar-beta} to
        $\boldsymbol{S}_k^{-1/2}\nabla f(\boldsymbol{x}_k)/
        \|\boldsymbol{S}_k^{-1/2}\nabla f(\boldsymbol{x}_k)\|$, we obtain Equations~\eqref{eq:guided-alignment-distribution}. Moreover,
        Inequality~\eqref{ineq:shape-bounds-Sk} implies
        $\boldsymbol{S}_k^{-1}\succeq
        (1/\overline\lambda_S)\boldsymbol{I}_n$.
        Thus, whenever $\|\nabla f(\boldsymbol{x}_k)\|>\epsilon$,
        \begin{equation*}
            \mathbb P(\Xi_k\ge\overline\tau_\epsilon\mid\mathcal F_{k-1})
            \ge\mathbb P(Z\ge\overline\lambda_S\overline\tau_\epsilon)
            =\mathbb P(Z\ge\vartheta p/n)=\pi_\epsilon,
        \end{equation*}
        which verifies Assumption~\ref{ass:mean-xi}.
    \end{proof}
    
    Proposition~\ref{prop:metric-guided-beta} shows how the conditional mean alignment depends on the direction of the gradient relative to the ellipsoid. For example, suppose that the nonzero gradient is parallel to an eigenvector of $\boldsymbol{S}_k$ with eigenvalue $\lambda$. In this case, multiplication by $\boldsymbol{S}_k^{-1}$ divides the gradient by $\lambda$, so the conditional mean formula gives $\mathbb E[\Xi_k\mid\mathcal F_{k-1}]=p/(n\lambda)$. Since a smaller eigenvalue corresponds to a longer ellipsoid axis, the mean alignment is larger when the gradient points along a longer axis.  The proposition also gives the same uniform probability bound as in the Haar case, so choosing a sufficiently small fixed $\vartheta>0$ ensures Inequality~\eqref{eq:zero-probability-condition}.

    \begin{remark}
        For fixed $n$ and $p$, choose $\vartheta$ so that Inequality~\eqref{eq:zero-probability-condition} holds. Then,  for either sampling rule, $\overline\tau_\epsilon$ and $\pi_\epsilon$ are independent of $\epsilon$, and Equations \eqref{eq:recovery-scale} and \eqref{eq:zero-analysis-scale} give $\Delta_\epsilon=\Theta(\epsilon)$ as $\epsilon\downarrow0$.  Therefore, Theorem~\ref{thm:zero-radius-complexity} gives $\mathbb E[K_\epsilon]=\mathcal O(\epsilon^{-2})$.
    \end{remark}

\section{Numerical experiments}\label{sec:numexp}\label{num:ten-start}
    In this section, we compare four practical RETRO variants with QARSTA~\cite{chen2024qfully} on low-rank projections of 126 test functions in $\mathbb R^{1000}$. The experiments examine how directional information stored in the ellipsoid can reduce function evaluation costs.  In particular, we compare the two sampling strategies introduced in Subsection~\ref{subsec:subspacesampling} and investigate the benefit of retaining a longest ellipsoid axis in the selected subspace. We first consider projection rank $r=5$, equal to the solver subspace dimension $p=5$, and then increase $r$ while keeping $p$ fixed.

\subsection{Test problems and solver settings}
    We select 126 unconstrained problems from CUTEst~\cite{gould2015cutest} (96 problems) and YATSOp~\cite{yatsop} (30 problems), all with dimension $n=1000$. These comprise all problems in these two collections that admit this dimension and have finite objective values at the prescribed starting point $\boldsymbol x_0$. For each base function $f$ and its supplied starting point $\boldsymbol x_0$, we orthonormalize a Gaussian matrix to obtain $\boldsymbol U\in\mathbb R^{1000\times r}$ and define the projected objective function
    \begin{equation*}
        f_r(\boldsymbol x) := f\bigl(\boldsymbol x_0+\boldsymbol U\boldsymbol U^\top(\boldsymbol x-\boldsymbol x_0)\bigr),~~~\text{where}~~~\boldsymbol U^\top\boldsymbol U=\boldsymbol I_r.
    \end{equation*}
    This function is constant along directions orthogonal to $\mathrm{col}(\boldsymbol U)$ and has effective dimension at most $r$. This construction gives a controlled instance of the low-effective-dimensional structure studied in random-embedding and active-subspace methods \cite{cartis2022dimensionality,constantine2015active,wang2016bayesian}. It allows us to vary the number of potentially active directions while keeping the ambient dimension and the base functions fixed.

    For each base problem and each projection rank, we use three random seeds to generate three projected objectives. Each objective is fixed throughout a run and shared by all five solvers, giving $126\times3=378$ cases for each rank. The solvers start from $\boldsymbol x_0$ and have no access to $\boldsymbol U$. Thus, any benefit from the fixed active subspace must be obtained through function evaluations. This setting is useful for examining RETRO: the ellipsoid updates in Subsection~\ref{subsec:ellipsoidupdate} retain information from successive subspace steps, and the sampling rules can use this information when choosing later subspaces.

    Table~\ref{tab:num-solvers} summarizes the five solvers. RETRO-H and RETRO-G use the Haar and metric-guided sampling rules, respectively, from Subsection~\ref{subsec:subspacesampling}. Both use the ellipsoidal updates of Algorithm~\ref{alg:main}. RETRO-H therefore provides a comparison with ball-based QARSTA under isotropic sampling, while the comparison between RETRO-H and RETRO-G examines the additional benefit of using the shape matrix to guide subspace selection.

    The other two RETRO variants add a heuristic that includes a longest axis of the current ellipsoid in the selected $p$-dimensional subspace. Such an axis corresponds to an eigenvector of $\boldsymbol S_k$ associated with its smallest eigenvalue. Since the shape matrix is updated according to previous iteration outcomes, this axis may provide a useful direction for subsequent steps. We denote these variants by RETRO-H+LA and RETRO-G+LA and treat longest-axis retention as a heuristic addition to the two sampling rules.
    \begin{table}[!ht]
    \centering
    \caption{QARSTA~\cite{chen2024qfully} and four RETRO variants. H denotes Haar sampling; G denotes metric-guided sampling using $\boldsymbol S_k^{-1/2}$; LA denotes retention of a longest ellipsoid axis within the $p$-dimensional subspace. The RETRO variants otherwise share the same models and update rules.}
    \label{tab:num-solvers}
    \begin{tabular}{@{}llll@{}}
    \toprule
        Solver & Trust region & Random directions & Retained axis\\
        \midrule
        QARSTA & Ball & Isotropic & None\\
        RETRO-H & Ellipsoid & Isotropic & None\\
        RETRO-G & Ellipsoid & Metric-guided & None\\
        RETRO-H+LA & Ellipsoid & Isotropic & Longest\\
        RETRO-G+LA & Ellipsoid & Metric-guided & Longest\\
    \bottomrule
    \end{tabular}
    \end{table}
    \FloatBarrier

    Each run has a budget of 5000 function evaluations. All solvers use determined quadratic interpolation models, $p=5$, $\Delta_0=0.1$, $\eta_1=0.1$, and $\eta_2=0.7$. The stopping thresholds are $\Delta_{\min}=10^{-8}$ for QARSTA and $\mathrm{radi}_{\min}=10^{-8}$ for RETRO. The contraction and expansion factors are $0.5$ and $2$, respectively. In particular, for RETRO, $\gamma_{\mathrm{dec}}^{\mathrm{sub}}=\gamma_{\mathrm{dec}}^{\mathrm{ful}}=0.5$ and $\gamma_{\mathrm{inc}}^{\mathrm{sub}}=\gamma_{\mathrm{inc}}^{\mathrm{ful}}=2$. We set $\boldsymbol S_0=\boldsymbol I_n$, $\underline\lambda_S=0.01$, and $\overline{\lambda}_S=1$. Hence, all solvers start with the same ball trust region, and the RETRO variants adapt its shape during the run.

    We compare the solvers using performance profiles~\cite{dolan2002benchmarking} based on function evaluation counts. For case $j$, let $f_{\mathrm{ref},j}$ be the smallest finite objective value observed across all five solvers within the evaluation budget. For solver $s$, let $t_{j,s}(\tau)$ be the first function evaluation count at which
    \begin{equation*}
        f_r(\boldsymbol x)\le f_{\mathrm{ref},j} +\tau\bigl(f_r(\boldsymbol x_{0,j})-f_{\mathrm{ref},j}\bigr),
    \end{equation*}
    where $\tau\in(0,1)$ is the desired accuracy level. Smaller values of $\tau$ give stricter targets relative to the best value observed. If the target is not reached within the evaluation budget, we set $t_{j,s}(\tau)=\infty$. For each solver $s$, its performance profile is
    \begin{equation*}
        \rho_s(\alpha;\tau):=\frac1N\left|\left\{j:r_{j,s}(\tau)\le\alpha\right\}\right|,~~~\text{where}~~~r_{j,s}(\tau):=\frac{t_{j,s}(\tau)}{\min_\ell t_{j,\ell}(\tau)}~~~\text{and}~~~N=378.
    \end{equation*}
    Thus, $\rho_s(1;\tau)$ is the proportion of cases on which solver $s$ reaches the target with the fewest function evaluations, including ties. For $\alpha>1$, the profile gives the proportion solved within $\alpha$ times the best evaluation count for each case.

\subsection{Results when $r=p=5$}
    Figure~\ref{fig:num-lowrank} presents the performance profiles for $r=p=5$ at the three accuracy levels $\tau=10^{-1},10^{-3},10^{-5}$. We first compare QARSTA, RETRO-H, and RETRO-G to examine the effect of using the ellipsoid in subspace selection. The profiles of QARSTA and RETRO-H nearly overlap, whereas RETRO-G generally lies above both, particularly for performance ratios close to one. Thus, for these problems, using the shape matrix to guide sampling gives a clearer improvement than changing the trust-region shape under isotropic sampling. This is consistent with the role of $\boldsymbol S_k$ in Subsection~\ref{subsec:subspacesampling}: metric-guided sampling favors directions along which the current ellipsoid permits longer steps.

    Retaining a longest axis gives a further improvement. At the low accuracy level $\tau=10^{-1}$, RETRO-H+LA and RETRO-G+LA attain the smallest evaluation count on 282 and 281 cases, respectively, corresponding to 74.6\% and 74.3\% of the 378 cases. The corresponding proportions are 27.5\% for both QARSTA and RETRO-H, and 38.6\% for RETRO-G. Ties are counted for every solver attaining the minimum, so these proportions need not sum to 100\%. The advantage also persists beyond $\alpha=1$. For example, both +LA variants solve 339 cases within $1.2$ times the best evaluation count, compared with 243, 242, and 264 cases for QARSTA, RETRO-H, and RETRO-G, respectively.

    At the higher accuracy levels, the +LA variants continue to perform best overall, although the differences between the profiles become smaller as $\alpha$ increases. Taken together, these results demonstrate an advantage of RETRO on problems with a fixed low-dimensional active subspace: the ellipsoid provides directional information that can be reused in subspace selection, reducing function evaluation costs without prior knowledge of the active subspace.
    \begin{figure}[htbp]
    \centering
    \includegraphics[width=\linewidth]{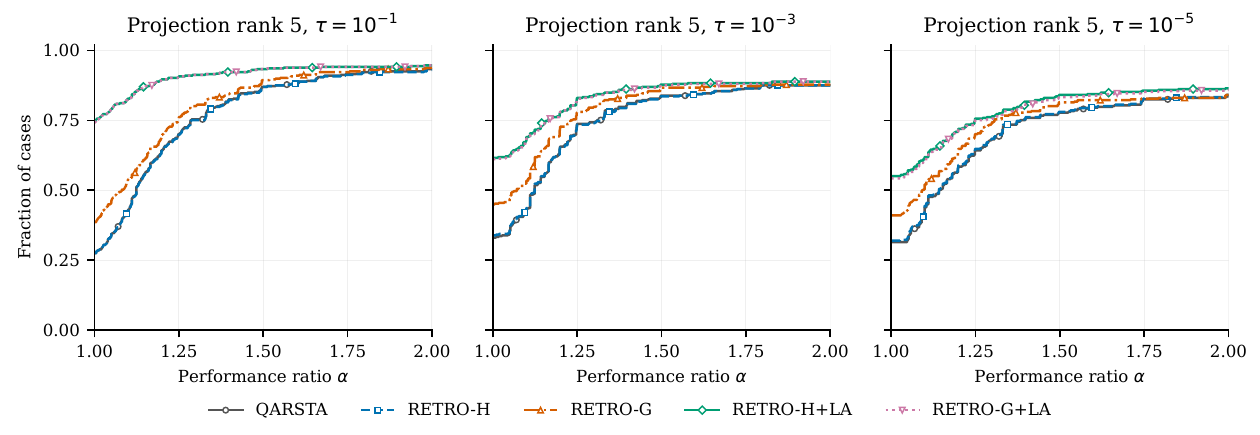}
    \caption{Performance profiles based on function evaluation counts for projection rank $r=5$, with $p=5$, $n=1000$, and 378 cases. The accuracy levels are $\tau=10^{-1},10^{-3},10^{-5}$ from left to right.}
    \label{fig:num-lowrank}
    \end{figure}
    \FloatBarrier

\subsection{Effect of the projection rank}
    We next repeat the comparison using the same 126 base functions with projection ranks $r=10,20,50$, keeping the solver subspace dimension $p=5$ and all other settings fixed. These tests examine how the benefits observed for $r=p$ change when each step is restricted to a subspace of smaller dimension than the projection space. Figure~\ref{fig:num-dimension-intermediate} presents the results.
    \begin{figure}[tbp]
    \centering
    \includegraphics[width=\linewidth]{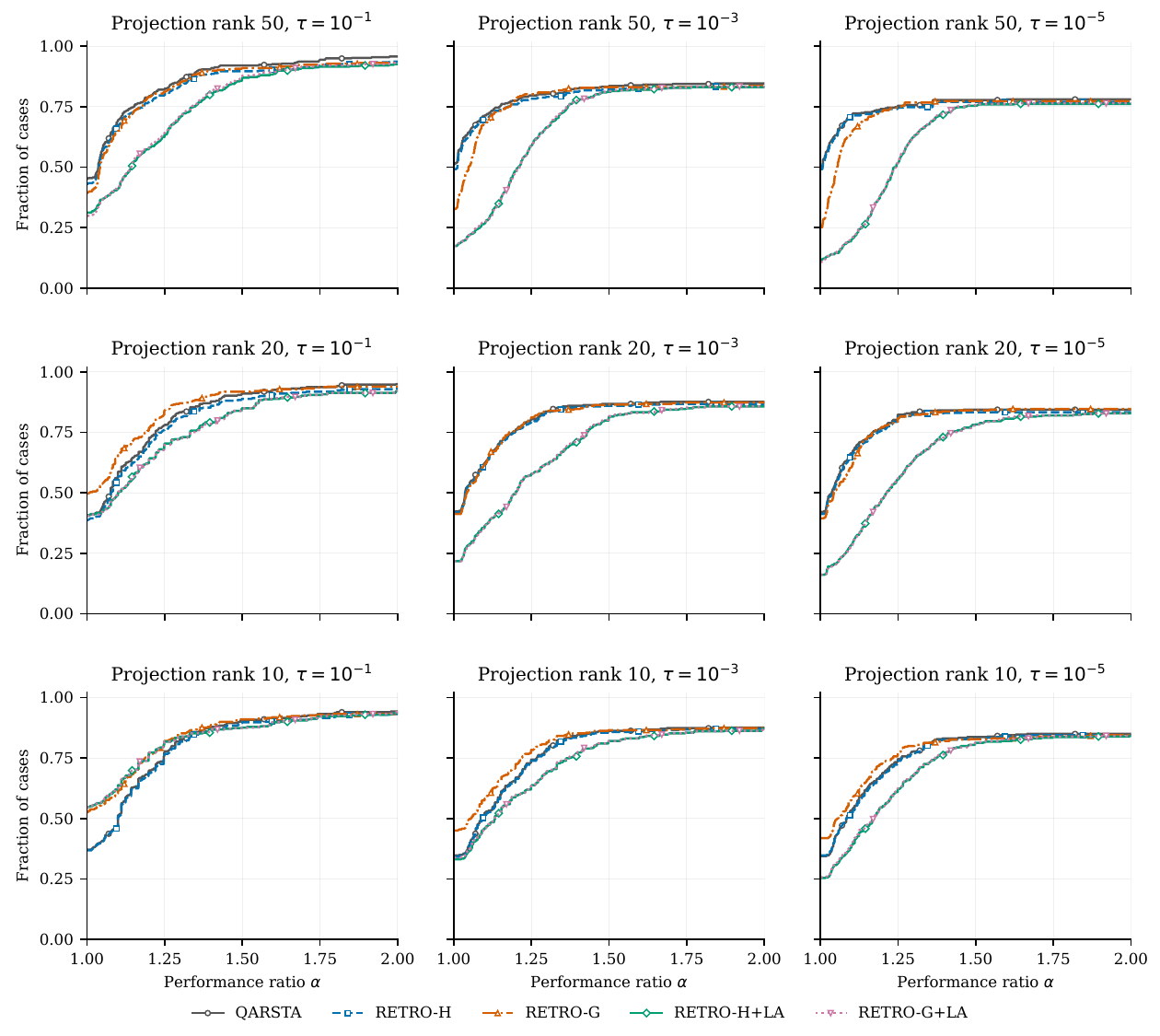}
    \caption{Performance profiles based on function evaluation counts for projection ranks $r=50,20,10$ from top to bottom, with $p=5$, $n=1000$, and 378 cases per rank. The accuracy levels are $\tau=10^{-1},10^{-3},10^{-5}$ from left to right.}
    \label{fig:num-dimension-intermediate}
    \end{figure}

    For $r=10$, retaining a longest axis still gives an advantage near $\alpha=1$ at the low accuracy level $\tau=10^{-1}$, but this advantage is lost at the two higher accuracy levels. RETRO-G remains competitive and generally improves on QARSTA and RETRO-H near $\alpha=1$. For $r=20$ and $r=50$, the +LA variants are less competitive, especially at the higher accuracy levels. The benefit of metric-guided sampling also becomes less pronounced: RETRO-G improves the low-accuracy profile for $r=20$, while its higher-accuracy profiles are close to those of QARSTA and RETRO-H; for $r=50$, it offers no consistent improvement.

    These results suggest that the usefulness of emphasizing previously explored directions depends on the projection rank relative to the solver subspace dimension. When more active directions must be explored with $p$ fixed, repeatedly including one longest axis may limit the benefit of the other directions in the sampled subspace. This provides a possible explanation for the reduced performance of the +LA variants as $r$ increases. The experiments therefore identify a setting in which RETRO is particularly effective: when the projection rank is close to the solver subspace dimension, using the ellipsoid to guide or retain directions can substantially reduce function evaluation costs. They also motivate adapting the strength of this directional preference, or the subspace dimension, when more active directions need to be explored.

\section{Conclusion}\label{sec:conclusion}
    In this paper, we proposed RETRO, a random-subspace trust-region framework for unconstrained derivative-free optimization. RETRO combines global changes in trust-region size with additional updates within the sampled subspace. We introduced an ellipsoidal subspace alignment measure and verified the required alignment properties for Haar-uniform and metric-guided sampling. Under the stated assumptions, we established almost-sure first-order liminf convergence and an $\mathcal{O}(\epsilon^{-2})$ expected iteration bound for reaching an iterate with gradient norm at most $\epsilon$.  Numerical experiments on projected low-rank problems in dimension $1000$ showed metric-guided sampling and longest-axis retention substantially reduced function evaluation costs relative to QARSTA when the projection rank matched the solver subspace dimension. These benefits became less consistent as the projection rank increased with the solver subspace dimension fixed.

\begin{acknowledgements}
    The authors sincerely thank Dr.~Warren Hare for helpful discussions and valuable suggestions regarding this work. 
\end{acknowledgements}

\bibliographystyle{siam}
\bibliography{references}

\begin{thebibliography}{10}

\bibitem{alarie2021two}
{\sc S.~Alarie, C.~Audet, A.~E. Gheribi, M.~Kokkolaras, and S.~Le~Digabel}, {\em Two decades of blackbox optimization applications}, EURO Journal on Computational Optimization, 9 (2021), p.~100011.

\bibitem{audet2017derivative}
{\sc C.~Audet and W.~Hare}, {\em Derivative-Free and Blackbox Optimization}, Springer, Cham, 2017.

\bibitem{audet2020model}
\leavevmode\vrule height 2pt depth -1.6pt width 23pt, {\em Model-based methods in derivative-free nonsmooth optimization}, Springer, Cham, 2020, pp.~655--691.

\bibitem{bandeira2014convergence}
{\sc A.~S. Bandeira, K.~Scheinberg, and L.~N. Vicente}, {\em Convergence of trust-region methods based on probabilistic models}, SIAM Journal on Optimization, 24 (2014), pp.~1238--1264.

\bibitem{cao2024first}
{\sc L.~Cao, A.~S. Berahas, and K.~Scheinberg}, {\em First- and second-order high probability complexity bounds for trust-region methods with noisy oracles}, Mathematical Programming, 207 (2024), pp.~55--106.

\bibitem{cartis2022dimensionality}
{\sc C.~Cartis and A.~Otemissov}, {\em A dimensionality reduction technique for unconstrained global optimization of functions with low effective dimensionality}, Information and Inference: A Journal of the IMA, 11 (2022), pp.~167--201.

\bibitem{cartis2023scalable}
{\sc C.~Cartis and L.~Roberts}, {\em Scalable subspace methods for derivative-free nonlinear least-squares optimization}, Mathematical Programming, 199 (2023), pp.~461--524.

\bibitem{cartis2026randomized}
\leavevmode\vrule height 2pt depth -1.6pt width 23pt, {\em Randomized subspace derivative-free optimization with quadratic models and second-order convergence}, Optimization Methods and Software,  (2026).

\bibitem{chen2018stochastic}
{\sc R.~Chen, M.~Menickelly, and K.~Scheinberg}, {\em Stochastic optimization using a trust-region method and random models}, Mathematical Programming, 169 (2018), pp.~447--487.

\bibitem{chen2026relationships}
{\sc Y.~Chen}, {\em Relationships between full-space and subspace quadratic interpolation models and simplex derivatives}, 2026.
\newblock \href{https://arxiv.org/abs/2602.10374}{arXiv 2602.10374}.

\bibitem{chen2026accuracy}
{\sc Y.~Chen, W.~Hare, and L.~Roberts}, {\em Accuracy and relationships of quadratic models in derivative-free optimization}, 2026.
\newblock \href{https://arxiv.org/abs/2605.12819}{arXiv 2605.12819}.

\bibitem{chen2024qfully}
{\sc Y.~Chen, W.~Hare, and A.~Wiebe}, {\em Q-fully quadratic modeling and its application in a random subspace derivative-free method}, Computational Optimization and Applications, 89 (2024), pp.~317--360.

\bibitem{chen2026clarsta}
\leavevmode\vrule height 2pt depth -1.6pt width 23pt, {\em {CLARSTA}: A random subspace trust-region algorithm for convex-constrained derivative-free optimization}, 2026.
\newblock \href{https://arxiv.org/abs/2506.20335}{arXiv 2506.20335}.

\bibitem{chikuse2003statistics}
{\sc Y.~Chikuse}, {\em Statistics on Special Manifolds}, vol.~174 of Lecture Notes in Statistics, Springer, New York, 2003.

\bibitem{conn2000trust}
{\sc A.~R. Conn, N.~I.~M. Gould, and P.~L. Toint}, {\em Trust-Region Methods}, SIAM, Philadelphia, 2000.

\bibitem{conn2009introduction}
{\sc A.~R. Conn, K.~Scheinberg, and L.~N. Vicente}, {\em Introduction to derivative-free optimization}, SIAM, 2009.

\bibitem{constantine2015active}
{\sc P.~G. Constantine}, {\em Active subspaces: {E}merging ideas for dimension reduction in parameter studies}, SIAM, 2015.

\bibitem{dolan2002benchmarking}
{\sc E.~D. Dolan and J.~J. Mor{\'e}}, {\em Benchmarking optimization software with performance profiles}, Mathematical Programming, 91 (2002), pp.~201--213.

\bibitem{dzahini2024stochastic}
{\sc K.~J. Dzahini and S.~M. Wild}, {\em Stochastic trust-region algorithm in random subspaces with convergence and expected complexity analyses}, SIAM Journal on Optimization, 34 (2024), pp.~2671--2699.

\bibitem{gould2015cutest}
{\sc N.~I.~M. Gould, D.~Orban, and P.~L. Toint}, {\em {CUTEst}: a constrained and unconstrained testing environment with safe threads for mathematical optimization}, Computational Optimization and Applications, 60 (2015), pp.~545--557.

\bibitem{gratton2015direct}
{\sc S.~Gratton, C.~W. Royer, L.~N. Vicente, and Z.~Zhang}, {\em Direct search based on probabilistic descent}, SIAM Journal on Optimization, 25 (2015), pp.~1515--1541.

\bibitem{hu2026matro}
{\sc W.~Hu, P.~Xie, Y.-X. Yuan, and L.~Zhang}, {\em {MATRO}: Metric-aware trust-region optimization with fully quadratic models}, 2026.
\newblock \href{https://arxiv.org/abs/2605.30845}{arXiv 2605.30845}.

\bibitem{larson2019derivative}
{\sc J.~Larson, M.~Menickelly, and S.~M. Wild}, {\em Derivative-free optimization methods}, Acta Numerica, 28 (2019), pp.~287--404.

\bibitem{liuzzi2019trust}
{\sc G.~Liuzzi, S.~Lucidi, F.~Rinaldi, and L.~N. Vicente}, {\em Trust-region methods for the derivative-free optimization of nonsmooth black-box functions}, SIAM Journal on Optimization, 29 (2019), pp.~3012--3035.

\bibitem{yatsop}
{\sc {POptUS}}, {\em {YATSOp}: Yet another test set for optimization}.
\newblock GitHub repository.
\newblock Accessed October 13, 2026.

\bibitem{wang2016bayesian}
{\sc Z.~Wang, F.~Hutter, M.~Zoghi, D.~Matheson, and N.~De~Freitas}, {\em Bayesian optimization in a billion dimensions via random embeddings}, Journal of Artificial Intelligence Research, 55 (2016), pp.~361--387.

\end{thebibliography}
\end{document}